\documentclass[a4paper]{article}
\usepackage[utf8]{inputenc}
\usepackage[T1]{fontenc}
\usepackage{amsmath}
\usepackage{amssymb}
\usepackage{amsthm}
\usepackage{verbatim}
\usepackage{enumitem}
\usepackage{hyperref}
\usepackage{a4wide}
\usepackage{caption}
\usepackage{subcaption}
\usepackage{float}
\usepackage{listings}
\usepackage{xcolor}
\usepackage{tikz}
\usetikzlibrary{arrows.meta, positioning, math, shapes}
\definecolor{codegreen}{rgb}{0,0.6,0}
\definecolor{codegray}{rgb}{0.5,0.5,0.5}
\definecolor{codepurple}{rgb}{0.58,0,0.82}
\definecolor{backcolour}{rgb}{0.95,0.95,0.92}
\lstdefinestyle{mystyle}{
    backgroundcolor=\color{backcolour},   
    commentstyle=\color{codegreen},
    keywordstyle=\color{magenta},
    numberstyle=\tiny\color{codegray},
    stringstyle=\color{codepurple},
    basicstyle=\ttfamily\footnotesize,
    breakatwhitespace=false,         
    breaklines=true,                 
    captionpos=b,                    
    keepspaces=true,                 
    numbers=left,                    
    numbersep=5pt,                  
    showspaces=false,                
    showstringspaces=false,
    showtabs=false,                  
    tabsize=2
}
\newtheorem{theorem}{Theorem}
\newtheorem{obs}[theorem]{Observation}
\newtheorem{clm}[theorem]{Claim}
\newtheorem{lemma}[theorem]{Lemma}
\newtheorem{prop}[theorem]{Proposition}
\newtheorem{defn}[theorem]{Definition}
\newtheorem{conj}[theorem]{Conjecture}
\newtheorem{cor}[theorem]{Corollary}
\newtheorem{question}[theorem]{Question}

\DeclareMathOperator{\ex}{ex}
\newcommand{\abs}[1]{\left\vert#1\right\vert}
\newcommand{\re}{\text{Re}\,}
\newcommand{\TT}[1]{\overrightarrow{T_{#1}}}
\newcommand{\CC}[1]{\overrightarrow{C\hspace{4pt}}_{\hspace{-5pt}#1}}

\newcommand{\PP}[1] {\overrightarrow{P\hspace{4pt}}_{\hspace{-5pt}#1}} 
\newcommand{\Cinv}[1]{\stackrel{\nrightarrow}{{C\hspace{4pt}}_{\hspace{-5pt}#1}}}
\newcommand{\oldStuff}[1]{}

\makeatletter\providecommand*{\diff}%
{\@ifnextchar^{\DIfF}{\DIfF^{}}}\def\DIfF^#1{%
\mathop{\mathrm{\mathstrut d}}%
\nolimits^{#1}\gobblespace}\def\gobblespace{%
\futurelet\diffarg\opspace}\def\opspace{%
\let\DiffSpace\!%
\ifx\diffarg(%
\let\DiffSpace\relax\else\ifx\diffarg[%
\let\DiffSpace\relax\else\ifx\diffarg\{%
\let\DiffSpace\relax\fi\fi\fi\DiffSpace}

\begin{document}

\title{
Generalized Tur\'an problem for directed cycles\thanks{The first, second and fourth authors were supported by the National Science Centre grant 2021/42/E/ST1/00193. The second and fourth authors were participants of the tutoring programme under the Excellence Initiative at the Jagiellonian University. The third author was supported by the Alexander von Humboldt Foundation in the framework of the Alexander von Humboldt Professorship of Daniel Kr\' a\v{l} endowed by the Federal Ministry of Education and Research.}
}

\author{
Andrzej Grzesik\thanks{Faculty of Mathematics and Computer Science, Jagiellonian University, {\L}ojasiewicza 6, 30-348 Krak\'{o}w, Poland. E-mail: {\tt andrzej.grzesik@uj.edu.pl}.}\and Justyna Jaworska\thanks{Faculty of Mathematics and Computer Science, Jagiellonian University, {\L}ojasiewicza 6, 30-348 Krak\'{o}w, Poland. E-mail: {\tt justynajoanna.jaworska@alumni.uj.edu.pl}.}\and
Bart\l{}omiej Kielak\thanks{Institute of Mathematics, Leipzig University, Augustusplatz 10, 04109 Leipzig, Germany. E-mail: {\tt bartlomiej.kielak@uni-leipzig.de}.}\and 
Piotr Kuc\thanks{Faculty of Mathematics and Computer Science, Jagiellonian University, {\L}ojasiewicza 6, 30-348 Krak\'{o}w, Poland. E-mail: {\tt piotr.kuc@student.uj.edu.pl}.}\and 
Tomasz \'Slusarczyk\thanks{Department of Mathematics, Massachusetts Institute of Technology, Cambridge, MA 02139, USA. E-mail: {\tt tomaszsl@alum.mit.edu}.}}

\date{}

\maketitle

\begin{abstract}
For integers $k, \ell \geq 3$, let $\ex(n, \CC{k}, \CC{\ell})$ denote the maximum number of directed cycles of length $k$ in any oriented graph on $n$ vertices which does not contain a directed cycle of length~$\ell$. We establish the order of magnitude of $\ex(n, \CC{k}, \CC{\ell})$ for every $k$ and $\ell$ and determine its value up to a lower error term when $k \nmid \ell$ and $\ell$ is large enough. Additionally, we calculate the value of $\ex(n, \CC{k}, \CC{\ell})$ for some other specific pairs $(k, \ell)$ showing that a diverse class of extremal constructions can appear for small values of $\ell$. 
\end{abstract}

\section{Introduction}

One of the fundamental questions in extremal graph theory, commonly referred to as the Tur\'an problem, is: `how many edges can an $n$-vertex graph have if it does not contain graph $F$ as a (not necessarily induced) subgraph?'. The asymptotic value of this number is known for many classes of graphs, e.g. cliques by Tur\'an's theorem~\cite{turan_thm} or non-bipartite graphs \cite{erdosstone}, but remains unknown for many bipartite graphs, e.g. even cycles~\cite{even_cycles}. 
 
Since Tur\'an's theorem was proved, many variants of the original problem were investigated. One straightforward variation is the following \textit{generalized Turán problem}~\cite{generalized_turan}: determining the maximum number of copies of a given graph $H$ in an $n$-vertex graph which does not contain a forbidden graph~$F$. We denote this quantity as $\text{ex}(n, H, F)$. 
Particular attention was paid to the case when both $H$ and $F$ are cycles -- starting with the Erd\H os pentagon conjecture \cite{erdos1984} on the value of $\text{ex}(n, C_5, C_3)$, through its solutions \cite{grzesik2012pentagon, c5_without_c3}, to general results on the order of magnitude~\cite{gishboliner2018} or exact values \cite{bekejanzer, gerbner2020generalized, grzesikkielak} for other cycle-lengths.
 
In this work we consider an analogous question in the setting of oriented graphs (directed graphs in which every pair of vertices is connected by at most one arc). More specifically, we focus on $\ex(n, \CC{k}, \CC{\ell})$ -- the maximum number of directed cycles of length $k$ in any $n$-vertex oriented graph which does not contain a subgraph isomorphic to a directed cycle of length $\ell$. 

In Section~\ref{sec:order_of_magnitude} we prove an exact order of magnitude of $\ex(n, \CC{k}, \CC{\ell})$ for every $k$ and $\ell$, by showing that if $k \nmid \ell$, then $\ex(n, \CC{k}, \CC{\ell}) = \Theta(n^{k})$, while if $k \mid \ell$, then $\ex(n, \CC{k}, \CC{\ell}) = \Theta(n^{k-1})$ (throughout this paper, when using asymptotic notation, we treat $k$ and $\ell$ of each subproblem as constants). We refer to the case $k \nmid \ell$ as the \emph{dense setting}, and the case $k \mid \ell$ as the \emph{sparse setting}. 

In the dense setting, a natural candidate for the extremal construction is a balanced blow-up of $\CC{k}$ -- it has $\Theta(n^k)$ copies of $\CC{k}$ and does not contain cycles of length not divisible by $k$, in particular $\CC{\ell}$. 
The leading constant can be further improved by reducing the number of blobs to the smallest divisor of $k$ greater than $2$ which does not divide $\ell$ (call this number $d$). We consider here only divisors greater than $2$ to enable the existence of a blow-up of $\CC{d}$.  
In Section \ref{sec:general_case} we prove that a balanced blow-up of $\CC{d}$ is indeed the best construction when $\ell$ is large enough (roughly at least~$2k^2$), but in the case $d \geq 5$ under an additional assumption that $k$ is odd or $\ell$ is even. 

This additional assumption is needed, because if $k$ is even and $\ell$ is odd, then we can consider an oriented graph that somewhat imitates an ``oriented blow-up of $\CC{2}$'' -- a random orientation of a complete balanced bipartite graph. This graph contains many copies of $\CC{k}$ when $2\mid k$ and no $\CC{\ell}$ when $2\nmid \ell$. However, it has less copies of $\CC{k}$ than a balanced blow-up of $\CC{d}$ if $d \leq 4$, that is why this construction is optimal for large enough $\ell$ if $d \geq 5$, $2\mid k$ and $2\nmid \ell$, as shown in Section~\ref{sec:random_bipartite_case}.

Then, in Section~\ref{sec:other_constructions}, we determine the value of $\ex(n, \CC{k}, \CC{\ell})$ for some particular values of $k$ and $\ell$, showing that when $\ell$ is comparable to $k$, then diverse extremal constructions can appear, including iterated blow-ups (for $\ex(n, \CC{4}, \CC{3})$), blow-ups of cycles with transitive tournaments in blobs (for $\ex(n, \CC{5}, \CC{3})$), or blow-ups of graphs other than cycles (for $\ex(n, \CC{5}, \CC{7})$). 

In Section~\ref{sec:sparse_case} we consider the sparse setting and determine the value of $\ex(n, \CC{3}, \CC{6})$ up to a lower order error term.

Finally, in Section~\ref{sec:directed}, we consider an analogous question in the setting of directed graphs. 
From the proven theorems for oriented graphs we derive an optimal bound, up to a lower order error term, for the maximum number of~$\CC{k}$ in a directed graph not containing $\CC{\ell}$ for large enough $\ell$ not divisible by $k$. 


\subsection{Preliminaries}

An \emph{oriented graph} is a pair $(V, E)$, where $V$ is a set of \emph{vertices} and $E \subset V\times V$ is a set of ordered pairs called \emph{arcs} with the property that for every $u,v \in V$ either $uv \notin E$ or $vu \notin E$. In particular, $vv \notin E$ for any $v \in V$. From now on, unless specified otherwise, we will refer to oriented graphs simply as graphs. For $v \in V$, by $N^+(v) = \{u \in V : vu \in E\}$ we denote the \emph{out-neighborhood} of $v$ and by $N^-(v) = \{u \in V : uv \in E\}$ the \emph{in-neighborhood} of~$v$.

The \emph{directed path} $\PP{k}$ is a~graph on vertex set $v_1, v_2, \ldots, v_k$ with $k-1$ arcs $v_1v_2, v_2v_3, \ldots, v_{k-1}v_k$. We say that such a path \emph{starts} in $v_1$, \emph{ends} in $v_k$, while $v_2, \ldots, v_{k-1}$ are its \emph{internal vertices}. The \emph{length} of a~directed path is the number of its arcs.
We denote by $N^+_i(v)$ the set of end vertices of directed paths of length $i$ starting in $v$. $N^-_i(v)$ is defined analogously.
The \emph{directed cycle} $\CC{k}$ is a graph on vertex set $v_1, v_2, \ldots, v_k$ with $k$ arcs $v_1v_2, v_2v_3, \ldots, v_kv_1$.
The graph $\Cinv{k}$ is a graph on vertex set $v_1, v_2, \ldots, v_k$ and $k$ arcs $v_1v_2, v_2v_3, \ldots, v_{k-1}v_k, v_1v_k$. Note that in $\Cinv{k}$ one arc is oriented differently than in $\CC{k}$. In particular, the graph $\Cinv{3}$ is the transitive tournament on $3$ vertices, and is therefore sometimes referenced as a \emph{transitive triangle} $\TT{3}$.
The graph $\overrightarrow{K}_{n,n}$ is an oriented graph with vertex set $A \cup B$, where $|A| = |B| = n$, containing all arcs $ab$ for any $a \in A$ and $b\in B$.

A \emph{blow-up} of a graph $G$ on $k$ vertices $v_1, v_2, \ldots, v_k$ is any graph $G'$ on a vertex set partitioned into disjoint independent sets $V_1, V_2, \ldots, V_k$, in which there is an arc $uv$ for $u \in V_i$ and $v \in V_j$ if and only if there is an arc $v_iv_j$ in $G$. The sets $V_i$ are called \emph{blobs}. A blow-up is \emph{balanced} if $||V_i| - |V_j|| \leq 1$ for any $i, j$. Note that for every integer $n \geq k$ there exists a balanced blow-up of $G$ on $n$ vertices. 

By a \emph{copy} of a graph $H$ in $G$ we mean a subgraph of $G$ isomorphic to $H$. Observe that a set of $|V(H)|$ vertices in $G$ may contain more than one copy of $H$. 
As mentioned in the Introduction, by $\ex(n, H, F)$ we denote the maximum number of copies of $H$ among $n$-vertex graphs not containing a copy of $F$. 
Similarly, by $\ex(n, H, \{F_1, F_2, \ldots, F_k\})$ we denote the maximum number of copies of $H$ among $n$-vertex graphs which contain no copies of $F_1, \ldots, F_k$. 

An oriented graph $H$ is a \emph{homomorphic image} of an oriented graph $F$ if there exists a~surjective map $\varphi: V(F) \rightarrow V(H)$ such that $\varphi(u')\varphi(v') \in E(H)$ for every arc $u'v' \in E(F)$ and for every arc $uv \in E(H)$ there exists an arc $u'v' \in E(F)$ satisfying $\varphi(u')\varphi(v') = uv$.
Note that if $H$ is a homomorphic image of $F$, then a sufficiently large blow-up of $H$ contains $F$ as a subgraph.

The following observation is a direct consequence of the directed version of Szemerédi's Regularity Lemma \cite{directed_szemeredi} and the blow-up lemma \cite{blowup_lemma}.

\begin{obs}\label{homomorphic_img_removal} 
For any oriented graphs $F$ and $G$, if $G$ does not contain $F$ as a subgraph, then by removing $o(n^2)$ arcs one can guarantee that $G$ does not contain any homomorphic image of $F$.
\end{obs}

Oftentimes, when calculating the number of copies of $H$ in an $n$-vertex graph $G$, we count the number of homomorphisms from $H$ to $G$ rather than subgraphs of $G$ isomorphic to $H$, because they are much simpler to calculate. This makes no difference when we are interested in a value up to $o(n^{|V(H)|})$ term, since the fraction of homomorphisms that are not injective is of order $o(n^{|V(H)|})$.

In the proofs, we often use a corollary from the following theorem of Brauer \cite{brauer1942problem} generalizing Sylvester's solution \cite{Syl82} of the well-known Frobenius Coin Problem.

\begin{theorem}[Brauer \cite{brauer1942problem}] \label{thm:brauer}
    Let $a_1, a_2, \ldots, a_k$ be positive integers that satisfy $\gcd (a_1, a_2, \ldots, a_k) = 1$. Denote $d_i = \gcd (a_1, \ldots, a_i)$. Then any integer $$\ell > a_2 \frac{d_1}{d_2} + a_3 \frac{d_2}{d_3} + \ldots + a_k \frac{d_{k-1}}{d_k} - \sum_{i=1}^{k}a_i$$ can be represented as $x_1a_1 + x_2a_2 + \ldots + x_ka_k$ for non-negative integers $x_1, \ldots, x_k$.
\end{theorem}

\begin{cor}\label{cor:coin_problem}
    For positive integers $a_1, a_2, \ldots, a_k$ denote $d_i = \gcd (a_1, \ldots, a_i)$. Then any integer $\ell$ divisible by $d_k$ and satisfying $$\ell > a_2 \frac{d_1}{d_2} + a_3 \frac{d_2}{d_3} + \ldots + a_k \frac{d_{k-1}}{d_k} - \sum_{i=1}^{k}a_i$$ can be represented as $x_1a_1 + x_2a_2 + \ldots + x_ka_k$ for non-negative integers $x_1, \ldots, x_k$.
    In particular, when $k=2$, any integer $\ell$ divisible by $d_2 = \gcd(a_1,a_2)$ and satisfying $\frac{\ell}{d_2} \ge \left(\frac{a_1}{d_2}-1\right)\left(\frac{a_2}{d_2}-1\right)$ can be expressed as $\ell = a_1x_1+a_2x_2$ for some non-negative integers $x_1$ and $x_2$.
\end{cor}

\begin{proof}
We apply Theorem~\ref{thm:brauer} for integers $\frac{a_1}{d_k}, \ldots, \frac{a_k}{d_k}$ obtaining that all integers greater than 
$$\frac{a_2}{d_k} \frac{d_1}{d_2} + \frac{a_3}{d_k} \frac{d_2}{d_3} + \ldots + \frac{a_k}{d_k} \frac{d_{k-1}}{d_k} - \sum_{i=1}^{k}\frac{a_i}{d_k}$$ can be represented as $x_1\frac{a_1}{d_k} + x_2\frac{a_2}{d_k} + \ldots + x_k\frac{a_k}{d_k}$ for non-negative integers $x_1, \ldots, x_k$.
Multiplying by $d_k$ we obtain the conclusion of the corollary.
\end{proof}

In everything that follows, when using the asymptotic notation $o$, $\mathcal{O}$ and $\Theta$ we treat $k$ and $\ell$ as constants.

\section{Order of magnitude}\label{sec:order_of_magnitude}

In this section we show the following theorem determining the order of magnitude of $\ex(n, \CC{k}, \CC{\ell})$ for every $k$ and $\ell$. 

\begin{theorem}\label{thm:order_of_magnitude}
Let $k, \ell \geq 3$ be different integers. If $k \nmid \ell$, then $\ex(n, \CC{k}, \CC{\ell}) = \Theta(n^{k})$, while if $k \mid \ell$, then $\ex(n, \CC{k}, \CC{\ell}) = \Theta(n^{k-1})$.
\end{theorem}

\begin{proof}
Clearly, the number of copies of $\CC{k}$ in an $n$-vertex graph is upper bounded by $n^k$. Since a balanced blow-up of $\CC{k}$ (see Figure \ref{fig:constr_dense}) does not contain any cycles of length not divisible by $k$ and contains asymptotically $\left( \frac{n}{k}\right)^k$ copies of $\CC{k}$, we immediately obtain the first part of the theorem. 

For the second part, note that a blow-up of $\CC{k}$ with one blob of size $1$ and remaining blobs balanced (see Figure \ref{fig:blowup_constr_sparse}) contains no cycles of length different than $k$ and has asymptotically $\left(\frac{n-1}{k-1}\right)^{k-1}$ copies of $\CC{k}$. This gives the lower bound for the second part of the theorem.

\begin{figure}[ht]
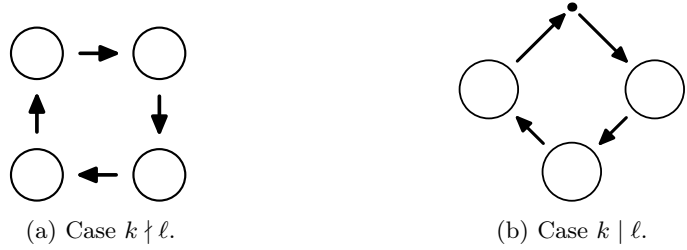

\centering
\begin{subfigure}{.4\textwidth}
  \centering
  \includegraphics[width=.4\linewidth]{drawings/blowup-empty.mps}
  \caption{Case $k \nmid \ell$.}
  \label{fig:constr_dense}
\end{subfigure}
\begin{subfigure}{.2\textwidth}
\end{subfigure}
\begin{subfigure}{.4\textwidth}
  \centering
  \includegraphics[width=.5\linewidth]{drawings/construction-sparse.mps}
  \caption{Case $k \mid \ell$.}
  \label{fig:blowup_constr_sparse}
\end{subfigure}

\caption{Constructions providing lower bounds for Theorem \ref{thm:order_of_magnitude}.}
\end{figure}

It remains to show that for any integers $k \geq 3$ and $t \geq 2$ it holds $\ex(n, \CC{k}, \CC{kt}) = \mathcal{O}(n^{k-1})$.
Let $G$ be any extremal graph on $n$ vertices, i.e., an oriented graph not containing $\CC{kt}$ and having the maximum possible number of $\CC{k}$. 

We say that an arc of $G$ is \emph{thick} if there exist at least $k^2tn^{k-3}$ different copies of $\CC{k}$ containing this arc. Remove all arcs in $G$ which are not thick, as long as there exist such. If there are no arcs left, then we removed $\mathcal{O}(n^{k-1})$ copies of $\CC{k}$ from $G$ and the theorem holds. Therefore, we can assume that every arc in $G$ is thick. 

The proof follows by induction on $t$. Firstly, we consider the case $t \in \{2, 3, \ldots, k-1\}$.

\textbf{Case 1: $t \leq k-1$.}
Let $v, w \in V(G)$ be any two vertices of $G$. We claim that every two directed paths of length $t$ that start in $v$ and end in $w$ share at least one internal vertex. Assuming otherwise, let $P$ and $Q$ be directed paths of length $t$ that share only the endpoints. Since each arc of $G$ is thick, we can find $t$ copies $D_1, \ldots, D_{t}$ of $\CC{k}$, each containing exactly one arc of $P$ and not sharing other vertices with $P$, $Q$ and other cycles $D_1, \ldots, D_{t}$ (see Figure \ref{fig:oomsub1}). But this creates a~copy of $\CC{kt}$ in $D_1 \cup \ldots \cup D_{t} \cup Q$, which is a~contradiction.

In particular, if we pick any two vertices $v, w$ of $G$, then there exist at most $(t-1)^2 n^{t-2}$ directed paths of length $t$ from $v$ to $w$. To see this, just fix any directed path $P$ from $v$ to $w$ and note that any other such directed path must share at least one internal vertex with $P$.

Now, we can count copies of $\CC{k}$ in $G$ in the following way. First, we choose any two vertices $v, w \in V(G)$, then we pick a~directed path of length $t$ from $v$ to $w$, and finally we choose the remaining $k - t - 1$ vertices to close the cycle. This way, we will count at most
$$ n(n-1) (t-1)^2 n^{t-2} n^{k - t - 1} = \mathcal{O}(n^{k-1})$$
copies of $\CC{k}$. Since each copy of $\CC{k}$ was counted at least once, we get the desired upper bound.

\begin{figure}[ht]
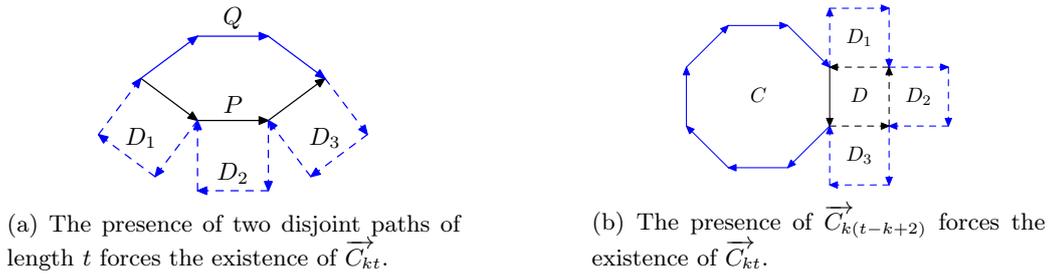

\centering
\begin{subfigure}{.4\textwidth}
  \centering
  \includegraphics[width=.6\linewidth]{drawings/oom-1.mps}
  \caption{The presence of two disjoint paths of length $t$ forces the existence of $\CC{kt}$.}
  \label{fig:oomsub1}
\end{subfigure}
\begin{subfigure}{.1\textwidth}
$ $
\end{subfigure}
\begin{subfigure}{.4\textwidth}
  \centering
  \includegraphics[width=.6\linewidth]{drawings/oom-2.mps}
  \caption{The presence of $\CC{k(t-k+2)}$ forces the existence of $\CC{kt}$.}
  \label{fig:oomsub2}
\end{subfigure}
\caption{Examples for $k=4$ and $t=3$, $t=4$ respectively.}
\label{fig:oom}
\end{figure}

\textbf{Case 2: $t \geq k$.}
Assume that $G$ contains a directed cycle $C$ of length $k(t-k+2)$. Then, since each arc of $G$ is thick, there exists a~copy $D$ of $\CC{k}$ containing exactly one arc of $C$ and not sharing other vertices of $C$, as well as there are copies $D_1, \ldots, D_{k-1}$ of $\CC{k}$, each containing exactly one arc of $D$ and not sharing other vertices of $C$, $D$ or other cycles $D_1, \ldots, D_{k-1}$ (see Figure \ref{fig:oomsub2}). 
But then, $C \cup D_1 \cup \ldots \cup D_{k-1}$ contains a~copy of $C_{kt}$ as $kt-k(k-2) - 1 + (k-1)(k-1) = kt$. This gives a~contradiction.

Since $G$ does not contain $\CC{k(t-k+2)}$, we get by induction that it has at most $\mathcal{O}(n^{k-1})$ copies of $\CC{k}$, which finishes the proof.
\end{proof}

\section{The dense setting}\label{sec:dense_case}

When determining $\ex(n, \CC{k}, \CC{\ell})$ in the dense setting, our usual framework is to take any extremal graph $G$ and remove $o(n^2)$ arcs using Observation~\ref{homomorphic_img_removal} to delete from $G$ all homomorphic images of~$\CC{\ell}$, that is, closed directed walks of length $\ell$. Note that this way we remove only $o(n^k)$ copies of $\CC{k}$. Finally, we delete from $G$ all arcs and vertices that do not occur in any $\CC{k}$. We refer to a graph obtained in such a process as a \emph{$(k,\ell)$-cleared} graph. 

\begin{defn}\label{def:kl-cleared}
For any integers $k, \ell \geq 3$ such that $k \nmid \ell$, an oriented graph $G$ is $(k, \ell)$-\emph{cleared} if it does not contain any homomorphic image of $\CC{\ell}$ and every arc and vertex of $G$ occurs in some $\CC{k}$.
\end{defn}

The advantage of $(k, \ell)$-cleared graphs is that they often do not contain specific small subgraphs, which greatly simplifies the proofs.

\begin{obs}\label{obs:cleanup_t3}
For any integer $k$, neither a $(k, k+1)$-cleared graph nor a $(k, 2k-1)$-cleared graph contains a transitive triangle.
\end{obs}

\begin{proof}
Let us assume that a $(k, \ell)$-cleared graph contains a transitive triangle. Since each of its arcs occurs in $\CC{k}$, there are directed paths of length $k-1$ between the endpoints of the arcs. Together with the arcs of the transitive triangle, they form a homomorphic image of $\CC{k+1}$ and a homomorphic image of $\CC{2k-1}$, forbidden if $\ell=k+1$ or $\ell=2k-1$. This is visualized for $k=4$ in Figure~\ref{fig:T3}.
\end{proof}

\begin{figure}[ht]
    \centering
    \includegraphics[width=0.4\textwidth]{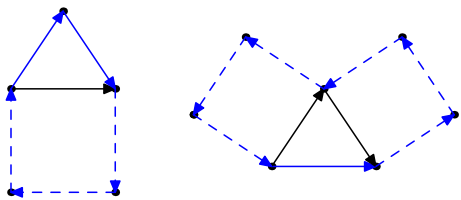}
    \caption{For $k = 4$, existence of $\TT{3}$ implies existence of homomorphic images of $\CC{5}$ and $\CC{7}$.}\label{fig:T3}
\end{figure}

\begin{obs}\label{obs:cleanup_small_cycles}
For any integers $m \geq 3$, $x \geq 1$ and $y \geq 0$, a $(k, mx+ky)$-cleared graph does not contain $\CC{m}$.
\end{obs}

\begin{proof}
Since every vertex of a $(k, mx+ky)$-cleared graph occurs in a $\CC{k}$, if the graph contains a $\CC{m}$, it also contains a $\CC{k}$ and a $\CC{m}$ joint by a vertex, which is a forbidden homomorphic image of $\CC{mx+ky}$.
\end{proof}

\subsection{The cycle blow-up case}\label{sec:general_case}

As mentioned in the Introduction, in this subsection we prove the following theorem, providing assumptions for $k$ and $\ell$ under which the value of $\ex(n, \CC{k}, \CC{\ell})$ is attained (up to a lower order error term) in a blow-up of a directed cycle. 

\begin{theorem}\label{thm:general_thm}
Let $k$ and $\ell$ be integers satisfying $k \geq 3$, $\ell \geq 2(k-1)^2$, and $k \nmid \ell$. 
Denote by $d$ the smallest integer greater than $2$ that divides $k$ but does not divide $\ell$. If $2\nmid k$, $2\mid \ell$ or $d \leq 4$, then $\ex(n, \CC{k}, \CC{\ell}) = \frac{n}{k}\left(\frac{n}{d}\right)^{k-1} + o(n^k)$.
\end{theorem}

The lower bound is achieved by a balanced blow-up of $\CC{d}$. 
Before we proceed to the proof of the upper bound, we show the following auxiliary lemma. 

\begin{lemma}\label{lem:ck_optimality}
Let $G$ be a graph satisfying the following property: for any copy $C$ of $\CC{k}$, each vertex of~$G$ has at most $\frac{2k}{d}$ neighbors in $C$. Then the number of copies of $\CC{k}$ in $G$ is at most $\frac{n}{k}\left(\frac{n}{d}\right)^{k-1}$.
\end{lemma}

\begin{proof}
The proof is based on a~method developed by Kr\'al', Norin, and Volec \cite{kral_norin_volec}. We say that a sequence $(v_i)_{i=0}^{k-1}$ of $k$ vertices is \emph{good} if $v_0 v_1\ldots v_{k-1}$ is a directed cycle in $G$. Note that there are $k$ different good sequences corresponding to a single $\CC{k}$.

For a~fixed good sequence $D = (v_i)_{i=0}^{k-1}$, we define the following sets:
\begin{align*}
A_0(D) & = V(G),\\
A_i(D) & = N^+(v_{i-1}) \textrm{ for } 1 \leq i \leq k-2,\\
A_{k-1}(D) & = N^+(v_{k-2}) \cap N^-(v_0).
\end{align*}

Define the~\emph{weight} $w(D)$ of a~good sequence $D$ as
\begin{align*}
w(D) = \prod_{i=0}^{k-1} \abs{A_i(D)}^{-1} = \frac 1n \prod_{i=1}^{k-1} \abs{A_i(D)}^{-1}.
\end{align*} 

This quantity has the following probabilistic interpretation. Suppose we want to sample $k$ vertices $w_0, \ldots, w_{k-1}$ as follows. We start with choosing $w_0$ uniformly at random from the set of all vertices of $G$. Next, we pick uniformly at random some out-neighbor of $w_0$ to be $w_1$. In general, $w_j$ is a~random vertex from the set $A_j((w_i)_{i=0}^{j-1})$. (Note that the definition of $A_j(D)$ depends only on the first $j$ elements of a~sequence~$D$.) Then, $w(D)$ is just the probability that the sequence $(w_i)_{i=0}^{k-1}$ obtained in this random process is equal to $D$.

In particular, the sum of the weights of all good sequences is at most one, since it is the sum of probabilities of pairwise disjoint events.

We similarly define sets $B_i(D)$ when vertices of a good sequence $D = (v_i)_{i=0}^{k-1}$ are sampled in the reversed order $v_0,v_{k-1},\ldots, v_1$:
\begin{align*}
B_0(D) & = V(G),\\
B_{1}(D) & = N^-(v_{0}),\\
B_{i}(D) & = N^-(v_{k-i+1}) \textrm{ for } 2 \leq i \leq k-2,\\
B_{k-1}(D) & = N^-(v_{2}) \cap N^+(v_0).
\end{align*}

Similarly, we define the corresponding weight 
\begin{align*}
\overline{w}(D) = \prod_{i=0}^{k-1} \abs{B_i(D)}^{-1} = \frac 1n \prod_{i=1}^{k-1} \abs{B_i(D)}^{-1}.
\end{align*} 

For a copy $\mathcal{D}$ of $\CC{k}$ in $G$ on vertices $v_0, \ldots, v_{k-1}$, we define corresponding good sequences $D_j = (v_{j}, v_{j+1}, \ldots, v_{j+k-1})$ for $j = 0,1,\ldots,k-1$, where the indices are taken modulo~$k$.

Denoting $n_{i,j} = \abs{A_i(D_j)}$ and $\overline{n}_{i,j} = \abs{B_i(D_j)}$ we have the following inequalities
\begin{align*}
    \frac{1}{2k} \sum_{j=0}^{k-1} (w(D_j) + \overline{w}(D_j)) & \geq \left(\prod_{j=0}^{k-1}w(D_j)\overline{w}(D_j)\right)^{1/2k} \\
    & = \frac 1n\left(\prod_{j=0}^{k-1}\prod_{i=1}^{k-1} n_{i,j}\overline{n}_{i,j}\right)^{-1/2k} \\
    & \geq \frac 1n\left(\frac{1}{2k(k-1)}\sum_{j=0}^{k-1}\sum_{i=1}^{k-1} \left(n_{i,j} + \overline{n}_{i,j}\right)\right)^{-(k-1)}.
\end{align*}
\begin{clm} \label{clm:vertex_contribution}
We have
\begin{align*}
    \sum_{j=0}^{k-1}\sum_{i=1}^{k-1} \left(n_{i,j} + \overline{n}_{i,j}\right) \leq 2k(k-1)\frac{n}{d}.
\end{align*}
\end{clm}
\begin{proof}
    For any vertex $u \in V(G)$, let $n_{i,j}(u) = 1$ if $u \in A_i(D_j)$ and $n_{i,j}(u) = 0$ otherwise. Similarly, let $\overline{n}_{i,j}(u) = 1$ if $u \in B_i(D_j)$ and $\overline{n}_{i,j}(u) = 0$ otherwise. Then,
 \begin{align*}
    \sum_{j=0}^{k-1}\sum_{i=1}^{k-1} \left(n_{i,j} + \overline{n}_{i,j}\right) = \sum_{u \in V(G)}\sum_{i=1}^{k-1}\sum_{j=0}^{k-1} \left(n_{i,j}(u) + \overline{n}_{i,j}(u)\right).
\end{align*}

For any fixed $1 \leq i < k-1$, the sum $\sum_{j=0}^{k-1}(n_{i,j}(u) + \overline{n}_{i,j}(u))$ is equal to the number of neighbors of $u$ in $\mathcal{D}$, hence is at most $\frac{2k}{d}$. Also, $\sum_{j=0}^{k-1} n_{k-1, j}(u)$ is not greater than the number of in-neighbors of $u$ in $\mathcal{D}$, and $\sum_{j=0}^{k-1} \overline{n}_{k-1, j}(u)$ is not greater than the number of out-neighbors of $u$ in $\mathcal{D}$, hence 
\[\sum_{j=0}^{k-1} (n_{k-1, j}(u) + \overline{n}_{k-1,j}(u)) \leq \frac{2k}{d}\]
as well. By summing over all choices of $u$ we get the desired inequality.
\end{proof}

Using Claim \ref{clm:vertex_contribution} above we obtain
\begin{align*}
    \frac{1}{2k} \sum_{j=0}^{k-1} (w(D_j) + \overline{w}(D_j)) & \geq \frac 1n \left( \frac{n}{d} \right)^{-(k-1)}
\end{align*}

Recall that $\sum_\mathcal{D}\sum_{j=0}^{k-1} w(D_j) \leq 1$ and $\sum_\mathcal{D}\sum_{j=0}^{k-1} \overline{w}(D_j) \leq 1$, where the first summation is over all copies of $\CC{k}$ in $G$. Denoting by $M$ the number of copies of $\CC{k}$ in $G$, by summing over all copies of $\CC{k}$ in $G$ we obtain
\begin{align*}
    \frac 1{k} \geq \frac 1n \left( \frac nd \right)^{-(k-1)} M.
\end{align*}
Hence
\begin{align*}
    M \leq \frac nk \cdot \left( \frac nd \right)^{k-1}
\end{align*}
as wanted.
\end{proof}

\begin{proof}[Proof of Theorem~\ref{thm:general_thm}]
Let $G$ be a $(k, \ell)$-cleared graph obtained from an extremal graph by losing $o(n^k)$ copies of $\CC{k}$ as explained in the beginning of Section~\ref{sec:dense_case}. To prove Theorem~\ref{thm:general_thm} it is enough to show that its assumptions imply that $G$ satisfies the property stated in Lemma~\ref{lem:ck_optimality}. We consider three separate cases depending on the assumed properties of $k$, $\ell$ and $d$. 

\paragraph{Assumption $2\nmid k$ or $2\mid \ell$.}
We start with showing that for any integer $3\leq m \leq d+1$, graph~$G$ does not contain a $\Cinv{m}$, that is, a cycle $\CC{m}$ with one arc oriented in the reverse direction. Assuming otherwise, there is an arc $xy$ and a directed path of length $m-1$ from $x$ to $y$. Since arc $xy$ is contained in $\CC{k}$, $G$ contains intersecting homomorphic images of $\CC{k}$ and $\CC{k+m-2}$. Consequently, a homomorphic image of a directed cycle of length $(k+m-2)x + ky$ for any integers $x, y \geq 0$. 
This will include a forbidden homomorphic image of $\CC{\ell}$ by Corollary~\ref{cor:coin_problem} if we show that $\ell$ is divisible by $g = \gcd(k,k+m-2)$ and large enough. If $g=1$, then trivially $g \mid \ell$. If $g=2$, then $2 \mid k$, so from our assumption $2\mid \ell$, as needed. Finally, if $g \geq 3$, then since $$g = \gcd(k+m-2, k) \leq m-2 < d,$$ from the definition of $d$ we get that $\ell$ is divisible by $g$. Therefore, since
\[\left(\frac{k + m - 2}{g} - 1\right)\left(\frac{k}{g} - 1\right) \leq \frac{1}{g^2}(k + m - 2 - 1)(k - 1) \leq \frac{1}{g}(2k-2)(k-1) \leq \frac{\ell}{g},\]
by Corollary~\ref{cor:coin_problem} we obtain that $G$ contains a homomorphic image of $\CC{\ell}$, which is a contradiction.

Now consider any directed path $P$ on at most $d$ vertices in $G$. If any vertex in $V(G) \setminus P$ has at least three neighbors in $P$, then it has at least two in- or out-neighbors in $P$, which creates a forbidden graph $\Cinv{m}$ with $3\leq m \leq d+1$. Thus, every vertex in $V(G) \setminus P$ has at most $2$ neighbors in $P$.
This means that for any directed cycle $C$ of length $k$, each vertex $v \in V(G)$ has at most $\frac{2k}{d}$ neighbors in $C$, because $C$ can be split into $\frac{k}{d}$ directed paths on $d$ vertices (or, if $v \in C$, into $\frac{k}{d}-1$ directed paths on $d$ vertices and one path on $d-1$ vertices).
Lemma \ref{lem:ck_optimality} implies the desired upper bound $\frac{n}{k}\left(\frac{n}{d}\right)^{k-1}$ on the number of copies of $\CC{k}$ in $G$.

\paragraph{Assumption $d = 3$.} 
We proceed similarly as in the previous case, by showing that $G$ does not contain any $\Cinv{3}$. Assuming otherwise, $G$ contains intersecting homomorphic images of $\CC{k}$ and $\CC{k+1}$, consequently a homomorphic image of a directed cycle of length $(k+1)x + ky$ for any integers $x, y \geq 0$. Since $k$ and $k + 1$ are relatively prime and $\ell \geq k(k-1)$, by Corollary~\ref{cor:coin_problem}, $G$~contains a homomorphic image of $\CC{\ell}$, which is a contradiction.

Similarly as in the previous case, we notice now that no vertex has more than $2$ neighbors on a directed path on at most $3$ vertices, because otherwise it creates a $\Cinv{3}$. Therefore, each vertex has at most $\frac{2k}{3}$ neighbors on any $\CC{k}$ in $G$, and by Lemma~\ref{lem:ck_optimality} we obtain the desired bound.

\paragraph{Assumption $d = 4$.}
The same argument as in the previous case shows that $G$ does not contain a $\Cinv{3}$. We will argue that $G$ also does not contain any $\CC{3}$. Assuming otherwise, $G$ contains intersecting $\CC{3}$ and $\CC{k}$, so consequently a homomorphic image of $\CC{3x + ky}$ for any integers $x, y \geq 0$. Since $d > 3$, either $3\nmid k$ or $3\mid \ell$, which means that $\gcd (3, k) \mid \ell$. Therefore, since $\ell > 2(k-1)$, by Corollary~\ref{cor:coin_problem} we obtain that $\ell$ is indeed of the form $3x + ky$, which is a contradiction. 

Knowing that $G$ does not contain any $\CC{3}$ or $\Cinv{3}$, it is trivial to observe that any vertex in $G$ has at most $\frac{2k}{d}=\frac{k}{2}$ neighbors in any $\CC{k}$ in $G$. This fact allows to apply Lemma~\ref{lem:ck_optimality} and obtain the wanted bound on the number of copies of $\CC{k}$ in $G$.
\end{proof}

\subsection{The random bipartite case}\label{sec:random_bipartite_case}

In this subsection, we prove that for large enough $\ell$ not divisible by $k$ in all cases not covered by Theorem~\ref{thm:general_thm} the value of $\ex(n,\CC{k},\CC{\ell})$ is attained (up to a lower order error term) in a random orientation of the complete balanced bipartite graph. This means that we assume that $k$ is even, $\ell$ is odd, and the smallest integer $d>2$ that divides $k$ but does not divide $\ell$ is equal to at least~$5$. The last condition is equivalent to saying that $4 \nmid k$ (to exclude $d=4$) and $3 \nmid k$ or $3 \mid \ell$ (to exclude $d=3$). Therefore, we will prove the following theorem. 

\begin{theorem} \label{thm:rand_thm}
    Let $k$ and $\ell$ be integers satisfying $k \geq 3$, $\ell > 33k^2$, and $k \nmid \ell$. If $2 \mid k$, $4 \nmid k$, $2\nmid \ell$, and $(3 \nmid k$ or $3\mid \ell)$, then $\ex(n, \CC{k}, \CC{\ell}) = \frac{2}{k}\left(\frac{n}{4}\right)^k + o(n^k)$.
\end{theorem}

The theorem is immediately implied by the following two lemmas. 

\begin{lemma} \label{lem:bipartite_optimal}
    Let $k$ and $\ell$ be integers satisfying $k \geq 3$, $\ell > 33k^2$, and $k \nmid \ell$. If $2 \mid k$, $4 \nmid k$, $2\nmid \ell$, and $(3 \nmid k$ or $3\mid \ell)$, then $\ex(n, \CC{k}, \CC{\ell})$ is achieved (up to a lower order error term) in an orientation of a complete bipartite graph.
\end{lemma}

\begin{lemma} \label{lem:knn_rand_optimal}
    Let $G$ be an orientation of a complete bipartite graph on $n$ vertices. If $2 \mid k$ and $4 \nmid k$, then $G$ contains at most $\frac{2}{k}\left(\frac{n}{4}\right)^k$ copies of $\CC{k}$.
\end{lemma}

\subsubsection{Proof of Lemma \ref{lem:bipartite_optimal}}

\begin{proof}
    Consider an extremal $(k, \ell)$-cleared graph $G$ on $n$ vertices, where $n$ is large enough (established later).
    The expected number of $\CC{k}$ in a random orientation of the complete balanced bipartite graph is equal to $\frac{2}{k}\left(\frac{n}{4}\right)^k - o(n^k)$, so $G$ contains at least $\frac{2}{k}\left(\frac{n}{4}\right)^k - o(n^k)$ copies of $\CC{k}$.

    For $v\in V(G)$, let $t_v$ denote the number of copies of $\CC{k}$ in $G$ containing $v$. Similarly, denote by $t_{u, v}$ the number of copies of $\CC{k}$ that contain both $u$ and $v$. 

    \begin{clm}
    For any $u, v \in V(G)$, it holds that $|t_v - t_u| < kn^{k-2} = o(n^{k-1})$. 
    \end{clm}
    \begin{proof}
        Assume for the sake of contradiction that $t_v - t_u \geq kn^{k-2}$. Define the graph $G'$ as $G$ after removing $u$ and replacing it with a copy $v'$ of $v$ (not adjacent to the original $v$). Since $G$ does not contain any homomorphic image of $\CC{\ell}$, $G'$ does not contain it either. $G'$ is also $(k, \ell)$-cleared, and the number of copies of $\CC{k}$ in $G'$ is larger than in $G$ by at least $(t_v - t_{u, v}) - t_u > 0$, as $t_{u, v} < kn^{k-2}$, which contradicts the maximality of $G$. 
    \end{proof}

    The above claim implies that for every vertex $v \in V(G)$, the number $t_v$ is not too far from the average number of $\CC{k}$ per vertex, which is greater than $\frac{1}{2}\left(\frac{n}{4}\right)^{k-1} + o(n^{k-1})$. Therefore, we may assume that $n$ is large enough, so that for every $v \in V(G)$ it holds that
    \begin{equation}\label{eq:t_v-bound}
        t_v \geq 0.45\left(\frac{n}{4}\right)^{k-1}.
    \end{equation}

    Let $\widetilde{G}$ be the underlying undirected graph of $G$, i.e., the undirected graph obtained from $G$ by removing the orientation of all arcs. The main idea of the proof is to show, using (\ref{eq:t_v-bound}), a lower bound on the minimum degree of $\widetilde{G}$ and that $\widetilde{G}$ does not contain odd cycles of small length. 

    We will then use following theorem to conclude that $\widetilde{G}$ (and therefore $G$) is bipartite.

    \begin{theorem}[Andrásfai-Erdös-Sós \cite{andrasfai-erdos-sos}]\label{thm:andrasfai-erdos-sos} Let $k \geq 1$ be an integer. If an $n$-vertex undirected graph $G$ contains no odd cycles of length at most $2k+1$ and has $\delta(G) > \frac{2}{2k+3}n$, then $G$ is bipartite.
    \end{theorem}

    We start by proving that $\widetilde{G}$ does not contain any short odd cycle.
    
    \begin{clm} \label{clm:no_short_odd_cycles}
        Graph $G$ does not contain any orientation of $C_m$ where $m \leq 23$ is odd.
    \end{clm}
    
    \begin{proof}
        Assume that $G$ contains such an orientation of a cycle with $a$ forward arcs and $b$ backward arcs, where $a + b = m \leq 23$ and $a>b$, and let $v$ be one of its vertices. Our goal is to upper bound~$t_v$.
        Note that if we build a $\CC{k}$ on each of the $b$ backward arcs we obtain a homomorphic image of $\CC{bk + (a - b)}$ containing $v$. 

        We will now show that the sets $N^+_i(v)$ and $N^+_j(v)$ for $i < j \leq k$ are disjoint whenever $j - i < 5$.
        Assuming otherwise, let $u \in N^+_i(v) \cap N^+_j(v)$. This implies that $G$ contains homomorphic images of $\PP{i+1}$ and $\PP{j+1}$ from $v$ to $u$. If we build a $\CC{k}$ on each arc of $\PP{i+1}$, then together with $\PP{j+1}$ it forms a homomorphic image of $\CC{ik + (j-i)}$. 
        Therefore, $v$ is contained in a $\CC{k}$, a homomorphic image of $\CC{ik + (j-i)}$ where $j - i < 5$, and in a homomorphic image of $\CC{bk + (a - b)}$ where $a - b \le 23$ and $a - b$ is odd.
        Note that  
        $$\gcd (k, bk + (a - b), ik + (j - i)) = \gcd (k, a - b,  j - i)$$         
        is less than $5$ and odd, so it is either 1 or 3. 
        If it is 3, then $3 \mid k$ and so $3 \mid \ell$ from assumptions of the lemma. 
        Applying Corollary~\ref{cor:coin_problem} for 
        \begin{align*}
        a_1 &= k,\\
        a_2 &= bk+(a-b) \leq bk+23-2b \leq 11k+1, \\
        a_3 &= ik + (j - i) \le ik + k-i \leq k^2-k+1,
        \end{align*}
        we have $d_1 = k$, $d_2 \leq a - b \leq 23$, and $d_3 = 1$ or $3$. Since
        \[a_2 \frac{d_1}{d_2} + a_3 \frac{d_2}{d_3} -a_1-a_2-a_3\leq (11k + 1)\cdot (k-1) + (k^2 -k+1) \cdot (23-1) - k \leq 33k^2 - 33k + 21 < \ell,\]
        the integer $\ell$ can be represented as a linear combination of $k$, $ik + (j - i)$, $bk + (a - b)$, so $G$ contains a homomorphic image of $\CC{\ell}$, which is a contradiction proving that indeed the sets $N^+_i(v)$ and $N^+_j(v)$ for $i < j \leq k$ are disjoint whenever $j - i < 5$. 

        The number $t_v$ is upper bounded by $|N^+_1(v)| \cdot |N^+_2(v)| \cdot \ldots \cdot |N^+_{k-1}(v)|$. Let $r$ be such a number in the set $\{2, 3, 4, 5, 6\}$ that $k\equiv r \pmod{5}$. Since $k \equiv 2 \pmod{4}$, we have $k+4r \equiv 10 \pmod{20}$, so $k\neq 2$ implies that $k\geq 30 - 4r$.
        Using the fact that every $5$ consecutive sets in $N^+_1(v), N^+_2(v), \ldots, N^+_{k-1}(v)$ are pairwise disjoint, we can group the product elements into $\frac{k-r}{5}$~groups of five and one group of size $r - 1$, and apply the AM-GM inequality in each group to obtain
        \[t_v \le \left(\frac{n}{5}\right)^{k-r} \left( \frac{n}{r - 1} \right)^{r - 1} < \left( \frac{n}{4} \right)^{k - (30 - 4r)} \left( \frac{n}{5} \right)^{30 - 5r} \left( \frac{n}{r - 1} \right)^{r - 1}.\]
        This is less than $0.45\left(\frac{n}{4}\right)^{k-1}$ whenever
        \[\left( \frac{1}{5} \right)^{30 - 5r} \left( \frac{1}{r - 1} \right)^{r - 1} < 0.45\left( \frac{1}{4} \right)^{30 - 4r - 1}.\]
        It is straightforward to verify that the last inequality holds for all $r\in \{2, 3, 4, 5, 6\}$. This contradicts inequality (\ref{eq:t_v-bound}) and concludes the proof.
        \end{proof}

        In order to prove a lower bound on the minimum degree of $\widetilde{G}$, we first upper bound the number of directed paths in $G$ of a given length. Let $p_i$ denote the number of copies of $\PP{i}$ in $G$. 

        \begin{clm}\label{cla:path_bound}
        For every positive integer $j$ it holds that $p_{2j} \leq n(\frac{n}{4})^{2j-1}$.
        \end{clm}
        \begin{proof}
            For any $v\in V(G)$ and integer $i$, denote by $p_i(v)$ the number of copies of $\PP{i}$ in $G$ ending in~$v$ and by $\alpha(v)$ the size of the largest independent set in $G$ which contains $v$. 
            The Cauchy-Schwarz inequality implies
            \begin{equation}\label{eq:p_i^2-bound}
            p_i^2(v) \leq \left(\sum_{u\in N^-(v)} p_{i-1}(u) \right)^2 \leq d^-(v) \sum_{u\in N^-(v)} p^2_{i-1}(u).
            \end{equation}      

            Since $G$ does not contain any triangles by Claim \ref{clm:no_short_odd_cycles}, for any vertices $u\in V(G)$ and $v \in N^+(u)$, the set $N^+(v) \cup N^-(v)$ is an independent set including $u$, hence 
            \begin{equation}\label{eq:d^+d^-bound}
            \forall_{v\in N^+(u)} \; d^+(v) + d^-(v) \leq \alpha(u).
            \end{equation}
            
            By applying the above two bounds and the AM-GM inequality the following inequalities hold for every integer $i$. 
            
            \begin{align*}
                \sum_v d^+(v) p_i^2(v) 
                &\stackrel{(\ref{eq:p_i^2-bound})}{\leq}
                \sum_v d^+(v) d^-(v) \sum_{u\in N^-(v)} p_{i-1}^2 (u) 
                = \sum_u \sum_{v \in N^+(u)} d^+(v) d^-(v) p_{i-1}^2 (u) \nonumber \\ 
                &\leq \sum_u \sum_{v \in N^+(u)} \left( \frac{d^+(v) + d^-(v)}{2}\right)^2 p_{i-1}^2(u) \\ 
                &\stackrel{(\ref{eq:d^+d^-bound})}{\leq}
                \sum_u d^+(u) \frac{\alpha^2(u)}{4} p^2_{i-1}(u) 
                \stackrel{(\ref{eq:p_i^2-bound})}{\leq}
                \sum_u d^+(u) \frac{\alpha^2(u)}{4} d^-(u) \sum_{v\in N^-(u)} p_{i-2}^2(v) \\ 
                &\leq \sum_u \left(\frac{d^+(u) + d^-(u) + \alpha(u)}{4}\right)^4 \sum_{v \in N^-(u)} p_{i-2}^2(v) \\ 
                &\leq \sum_u \left( \frac{n}{4} \right)^4 \sum_{v \in N^-(u)} p_{i-2}^2(v) 
                = \left( \frac{n}{4} \right)^4 \sum_v d^+(v) p_{i - 2}^2(v)
            \end{align*}
            Repeated application of the above bound yields for every integer $j$
            \begin{align*}
                \sum_v d^+(v) p_{2j-1}^2(v) 
                &\leq 
                \left( \frac{n}{4} \right)^4 \sum_v d^+(v) p_{2j-3}^2(v) 
                \leq 
                \left( \frac{n}{4} \right)^8 \sum_v d^+(v) p_{2j-5}^2(v) 
                \leq \dots  \\ 
                &\leq
                \left( \frac{n}{4} \right)^{4j-4} \sum_v d^+(v) p_1^2(v) = \left( \frac{n}{4} \right)^{4j-4} e(G).
            \end{align*}
            This, together with the Cauchy-Schwarz inequality, allows to finally bound $p_{2j}$ by counting each $\PP{2j}$ by its last arc as follows
            \begin{align*}
                p_{2j} &= \sum_{vu\in E(G)} p_{2j-1}(v) 
                \leq \left(e(G) \sum_{vu\in E(G)} p^2_{2j - 1}(v)\right)^{\frac{1}{2}} \\ 
                &= \left(e(G) \sum_v d^+(v) p^2_{2j-1}(v)\right)^{\frac12} 
                \leq \left(\frac{n}{4}\right)^{2j-2} e(G) 
                \leq n \left(\frac{n}{4}\right)^{2j-1},
            \end{align*}
            where the last inequality follows from Mantel's theorem. 
        \end{proof}

        \begin{clm} \label{clm:big_degree}
            Graph $\widetilde{G}$ satisfies $\delta(\widetilde{G}) > \frac{2}{23}n$.
        \end{clm}
        \begin{proof}
        Assume $\widetilde{G}$ contains a vertex $v$ with $d(v) \leq \frac{2}{23}n$. It holds that
        $$
        t_v \leq |N^+(v)| \cdot |N^+_2(v)| \cdot p_{k-4} \cdot |N^-(v)|
        $$
        because we can form a unique $\CC{k}$ containing $v$ by choosing the successor of $v$ on the cycle on $|N^+(v)|$ ways, then the next vertex on at most $|N^+_2(v)|$ ways, then the subsequent $k-4$ vertices together (they must form a path) and finally, the last vertex on at most $|N^-(v)|$ ways.
        
        Since $G$ does not contain any triangle by Claim~\ref{clm:no_short_odd_cycles}, the sets $N^+(v)$, $N^+_2(v)$, $N^-(v)$ must be pairwise disjoint and consequently $|N^+(v)| \cdot |N^-(v)| \cdot |N^+_2(v)| \leq \left(\frac{d(v)}{2}\right)^2 (n - d(v))$. Together with Claim~\ref{cla:path_bound} this yields
        \begin{align*}
        t_v \leq \left(\frac{d(v)}{2}\right)^2 (n - d(v)) \cdot n \left( \frac{n}{4} \right)^{k - 5} \leq \frac{21}{23^3}n^4\left( \frac{n}{4} \right)^{k - 5} = \frac{21\cdot4^4}{23^3}\left( \frac{n}{4} \right)^{k - 1} < 0.45\left( \frac{n}{4} \right)^{k - 1},
        \end{align*}
        contradicting (\ref{eq:t_v-bound}) and finishing the proof.
        \end{proof}

        As mentioned earlier, since $\delta(\widetilde{G}) > \frac{2}{23}n$ by Claim \ref{clm:big_degree} and $\widetilde{G}$ does not contain odd cycles of length at most 23 by Claim~\ref{clm:no_short_odd_cycles}, Theorem \ref{thm:andrasfai-erdos-sos} for $k=11$ implies that $\widetilde{G}$, and so $G$, is bipartite. 
        Moreover, we can assume that $G$ is an orientation of a complete bipartite graph, because adding an arc between the two color classes of $G$ does not create $\CC{\ell}$ and may only increase the number of copies of $\CC{k}$. 
\end{proof}

\subsubsection{Proof of Lemma \ref{lem:knn_rand_optimal}}

\begin{proof}
    Label the vertices of $G$ as $v_{1}, \ldots, v_{m}, v_{m+1}, \ldots, v_n$ where the first $m$ vertices belong to the first color class and the remaining $n - m$ vertices to the second color class. Let $M$ be the adjacency matrix of $G$. Then
    \begin{itemize}
        \item $M_{ij} = 0$ if $1 \leq i, j \leq m$ or $m+1 \leq i, j \leq n$.
        \item For $1 \leq i \leq m < j \leq n$, either $M_{ij} = 1$ or $M_{ji} = 1$.
    \end{itemize}

    The number of homomorphic images of $\CC{k}$ in $G$ is equal to $\frac{1}{k}\text{tr}(M^k) = \frac1k (\lambda_1^k + \ldots + \lambda_n^k)$, where $\{\lambda_1, \lambda_2, \ldots, \lambda_n\} = \sigma(M)$ is the multiset of eigenvalues of $M$. Assume the eigenvalues $\lambda_i$ are ordered so that $\re \lambda_1 \geq \re \lambda_2 \geq \ldots \geq \re \lambda_n$. Note the following properties:
    \begin{itemize}
        \item If $\lambda \in \sigma(M)$, then $\bar{\lambda} \in \sigma(M)$, because $M$ is a real matrix and so its characteristic polynomial has real coefficients.
        \item If $\lambda \in \sigma(M)$, then $-\lambda \in \sigma(M)$, because when $\overrightarrow{v} = (v_1, \ldots, v_n)$ is an eigenvector for the eigenvalue $\lambda$, then $\overrightarrow{v'} = (-v_1, \ldots, -v_m, v_{m+1}, \ldots, v_n)$ is an eigenvector for $-\lambda$. In particular, $\re \lambda_{n+1-i} = -\re \lambda_i$.
        \item By the non-negative variant of the Perron-Frobenius theorem (see e.g. \cite{minc1988}, Theorem 4.2), $\lambda_1$ is real, non-negative and for every $i$ it holds that $|\lambda_i| \leq \lambda_1$.
    \end{itemize}

    We start with proving a bound on the sum of the positive real parts of the eigenvalues of $M$. In the proof we will make use of the following proposition, which is a special case of Theorem 2 from \cite{fan1950theorem}.
        \begin{prop}[Ky Fan \cite{fan1950theorem}] \label{prop:eigenvalue_sum}
            If $\rho_1 \geq \rho_2 \geq \ldots \geq \rho_n$ are the eigenvalues of $\frac{M+M^T}{2}$, then $\sum_{i=1}^{\lfloor n / 2 \rfloor } \re \lambda_i \leq \sum_{i=1}^{\lfloor n / 2 \rfloor } \rho_i$.
        \end{prop}

    \begin{clm}
        $\sum_{i=1}^{\lfloor n / 2 \rfloor} \re \lambda_i \leq \frac{1}{2}\sqrt{m(n-m)}$
    \end{clm}

    \begin{proof}
        The matrix $A = \frac{M+M^T}{2}$ is a block matrix, where $A_{ij} = 0$ when $1\leq i,j \leq m$ or $m < i, j \leq n$ and $A_{ij} = \frac12$ otherwise. In other words, the matrix $2A$ is the adjacency matrix of $K_{m,n-m}$. Therefore, the eigenvalues of $A$ are
        $$
        (\rho_1, \rho_2, \ldots, \rho_n) = \left(\frac12\sqrt{m(n-m)}, 0, 0, \ldots, 0, -\frac12\sqrt{m(n-m)} \right)
        $$
        and Proposition \ref{prop:eigenvalue_sum} gives $\sum_{i=1}^{\lfloor n / 2 \rfloor } \re \lambda_i \leq \sum_{i=1}^{\lfloor n / 2 \rfloor } \rho_i = \frac12\sqrt{m(n - m)}$.
    \end{proof}

    We are now ready to bound $\text{tr}(M^k)$. Denoting $\rho = \lambda_1$, we have
    \begin{align*}
    \sum_{i=1}^n \lambda_i^k 
    &\stackrel{(a)}{=} \sum_{i=1}^{n} \re \lambda_i^k 
    \stackrel{(b)}{=} 2 \sum_{i=1}^{\lfloor n/2 \rfloor} \re \lambda_i^k 
    = 2 \left( \rho^k + \sum_{i=2}^{\lfloor n/2 \rfloor} \re \lambda_i^k \right)
    \stackrel{(c)}{\leq} 2\left( \rho^k + \sum_{i=2}^{\lfloor n/2 \rfloor} k \rho^{k-1} \re \lambda_i \right) \\
    &\stackrel{(d)}{\leq} 2\left(\rho + \sum_{i=2}^{\lfloor n/2 \rfloor} \re \lambda_i\right)^k 
    = 2 \left(\sum_{i=1}^{\lfloor n/2 \rfloor} \re \lambda_i\right)^k
    \leq 2 \left(\frac{1}{2}\sqrt{m(n-m)}\right)^k 
    \leq 2 \left( \frac{n}{4} \right)^k,
    \end{align*}
    where the respective inequalities follow from the following properties:
    \begin{enumerate}[label=($\alph*$)]
        \item Because $\text{tr}(M^k)$ is real.  
        \item As noted earlier, $\re \lambda_{n+1-i} = -\re \lambda_i$, so $\re \lambda_{n+1-i}^k = \re \lambda_i^k$ since $k$ is even.  
        \item It follows from the fact that for any complex number $z$ with $\re z \geq 0$ and $k\equiv 2 \pmod{4}$ we have $\re z^k \leq k |z|^{k-1} \re z$. The proof can be found in \cite{grzesik2023cycles}, but we include it for completeness. \linebreak
        Let $\alpha = \text{arg}\, (z)$. By symmetry, we can assume that $\alpha \in [0, \frac{\pi}{2}]$. Let $\beta = \frac{\pi}{2} - \alpha$. We first show that $-\cos k\beta \leq k \sin \beta$. Indeed, if $0\leq \beta \leq \frac{\pi}{2k}$, then $-\cos k\beta \leq 0 \leq k\sin \beta$, while if $\beta > \frac{\pi}{2k}$, then $\sin \beta > \frac{1}{k}$ and $-\cos k\beta \leq 1 < k \sin \beta$. Having this, we have
        $$
        \re z^k = |z|^k \cos k\alpha = -|z|^k \cos k\beta \leq k|z|^k \sin \beta = k|z|^k \cos \alpha = k|z|^{k-1} \re z.
        $$
        \item Binomial expansion of the right-hand side yields the left-hand side and possible additional terms, which are non-negative since $\re \lambda_i \ge 0$ for $i \le n/2$.
    \end{enumerate}

The proven inequality shows that the number of homomorphic images of $\CC{k}$ in $G$ is upper bounded by $\frac{2}{k}\left(\frac{n}{4}\right)^k$ as desired. 
\end{proof}

\subsection{Small cycles}\label{sec:other_constructions}

In this subsection, we show that the assumption that $\ell$ is large enough compared to $k$ in Theorems~\ref{thm:general_thm} and \ref{thm:rand_thm} is indeed necessary. We focus on the smallest cases $k = 3, 4, 5$ that already present the growing level of complexity and different possible extremal constructions.

\subsubsection{Maximizing $\CC{3}$}

\begin{theorem}\label{thm:ex_C_3_C_l}\hfill
\begin{enumerate}[label=\alph*\textnormal{)}]
\item $\ex(n,\CC{3},\CC{\ell})=\left\lceil\frac{n}{3}\right\rceil\left\lceil\frac{n-1}{3}\right\rceil\left\lceil\frac{n-2}{3}\right\rceil$ for $\ell=4,\ 5$.
\item $\ex(n, \CC{3}, \CC{\ell}) = \frac{n^3}{27} + O(n^2)$ for $\ell > 6$ not divisible by $3$.
\end{enumerate}
\end{theorem}
\begin{proof}
In both cases, the lower bound is achieved in a balanced blow-up of $\CC{3}$.

To prove the upper bound for \textit{a}), consider an extremal graph $G$. Since for $\ell=4$ and $\ell=5$ every homomorphic image of $\CC{\ell}$ is a copy of $\CC{\ell}$, we may assume that $G$ is $(3, \ell)$-cleared without removing any $\CC{3}$. 

Knowing this, the following claim implies part \textit{a}) of the theorem. 

\begin{clm}\label{clm:t3-free}
$\ex(n, \CC{3}, \TT{3}) = \left\lceil\frac{n}{3}\right\rceil\left\lceil\frac{n-1}{3}\right\rceil\left\lceil\frac{n-2}{3}\right\rceil$.
\end{clm}
\begin{proof}
Consider the undirected graph $\widetilde{G}$ obtained from $G$ by removing the orientation of all arcs. Note that $\widetilde{G}$ does not contain $K_4$, because every orientation of edges of~$K_4$ contains at least one copy of $\TT{3}$. Therefore, the number of copies of~$\CC{3}$ in $G$ is upper bounded by the number of copies of~$K_3$ in an undirected graph not containing~$K_4$. The latter was proven by Erd\H os \cite{erdos1962} to achieve the maximum when $\widetilde{G}$ is a~balanced blow-up of $K_3$. Since the edges of a blow-up of $K_3$ can be oriented in a way that each $K_3$ becomes $\CC{3}$, the claim follows.
\end{proof}

As we noticed earlier in Observation \ref{obs:cleanup_small_cycles}, a $(3,\,3k+4)$-cleared graph does not contain $\CC{4}$ and a $(3,\,3k+5)$-cleared graph does not contain $\CC{5}$, so the proven part \textit{a}) already implies that for $3 \nmid \ell$, $\ex(n, \CC{3}, \CC{\ell}) = \frac{n^3}{27} + o(n^3)$. However, the error term can be improved to $O(n^2)$ as stated in part~\textit{b}), which we will prove now.

Let $G$ be any (non-cleared) oriented graph on $n$ vertices that does not contain $\CC{\ell}$ and maximizes the number of copies of~$\CC{3}$. Perform the following procedure on $G$ -- as long as there exists an arc which is contained in at most $\ell$ copies of~$\CC{3}$, remove this arc. Note that in this way, we remove at most $\binom{n}{2}\ell = O(n^2)$ copies of~$\CC{3}$.

\begin{clm}
For every integer $m \geq 0$ such that $\ell-3m \geq 4$, graph $G$ does not contain $\CC{\ell-3m}$.
\end{clm}

\begin{proof}
We proceed by induction on $m$. The statement is true for $m=0$, because $G$ does not contain $\CC{\ell}$. For $m \geq 1$, let $C$ be any copy of $\CC{\ell-3m}$ in $G$ and $vw$ an arc of $C$. Since $vw$ is contained in at least $\ell$ copies of $\CC{3}$, there exists a~vertex $x \in V(G) \setminus V(C)$ such that $vwx$ is a~copy of $\CC{3}$. Using the same reasoning, we can find distinct vertices $y, z \in V(G) \setminus V(C)$ such that $xvy$ and $xzw$ are copies of $\CC{3}$. But then $C$ with arc $vw$ replaced by a directed path $vyxzw$ is a~copy of $\CC{\ell-3(m-1)}$ contradicting the induction assumption. 
\end{proof}

\begin{figure}[ht]
    \centering
    \includegraphics[width=0.25\textwidth]{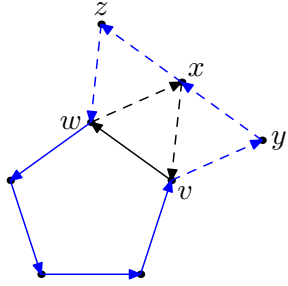}
    \caption{Example for $\ell = 8$, $m = 1$}
\end{figure}

This claim shows that graph $G$ satisfies the assumptions of part \textit{a}) of the theorem, hence
$\ex(n, \CC{3}, \CC{\ell}) \leq \frac{n^3}{27} + O(n^2).$ \end{proof}

\subsubsection{Maximizing $\CC{4}$}

In contrast to the case of $\CC{3}$, when we maximize the number of copies of $\CC{4}$, an extremal construction different than a blow-up of a cycle (or an orientation of a complete bipartite graph as for the case of large $\ell$) can occur.

An iterated blow-up of a graph $H$ is a graph obtained in the following procedure. We take a balanced blow-up of $H$, then place a balanced blow-up of $H$ on the vertices of each blob, and continue further placing balanced blow-ups of $H$ inside new blobs, until it is possible (see Figure~\ref{fig:Iterated}). 

\begin{figure}[H]
    \centering
    \includegraphics[width=0.22\textwidth]{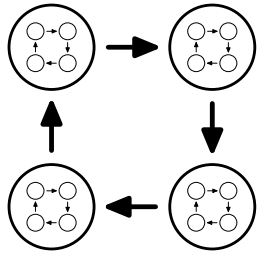}
    \caption{An iterated blow-up of $\CC{4}$.}\label{fig:Iterated}
\end{figure}

\begin{theorem}\label{thm:ex_c4_c3}
The value of $\ex(n, \CC{4}, \CC{3})$ is achieved in the iterated blow-up of $\CC{4}$. In particular,
$$
\ex(n, \CC{4}, \CC{3}) = 
\frac{n^4}{4^4 - 4} + o(n^4).
$$
\end{theorem}

\begin{proof}
In any graph not containing $\CC{3}$, each copy of~$\CC{4}$ is induced, thus the number of copies of $\CC{4}$ is upper bounded by the maximum number of induced copies of $\CC{4}$. The latter is claimed in~\cite{directed_cfour_density} (see also \cite{inducibility}) to be maximized in the iterated blow-up of $\CC{4}$.
\end{proof}

\begin{theorem}\label{thm:ex_c4_cl}
If $\ell > 4$ and $4 \nmid \ell$, then $\ex(n, \CC{4}, \CC{\ell}) = 
\left(\frac{n}{4}\right)^4 + o(n^4)$.
\end{theorem}

\begin{proof}

The balanced blow-up of $\CC{4}$ provides the lower bound $\left(\frac{n}{4}\right)^4 + o(n^4)$. The rest of the proof is dedicated to the upper bound.

Consider any extremal $(4, \ell)$-cleared graph $G$.
By Observation \ref{obs:cleanup_small_cycles} we know that it is enough to consider $5 \leq \ell \leq 7$, because for any larger $\ell$, we can use the fact that if $G$ contains $\CC{\ell - 4i}$, where $i$ is chosen so that $5 \leq \ell - 4i \leq 7$, then it contains a homomorphic image of $\CC{\ell}$. 

\paragraph{Cases $\ell = 5$ and $\ell = 7$.}

By Observation~\ref{obs:cleanup_t3}, any $(4,\,5)$-cleared graph and any $(4,\,7)$-cleared graph does not contain $\TT{3}$. This implies that each vertex in $V(G)$ has at most $2$ neighbors in any $\CC{4}$. Therefore, by Lemma \ref{lem:ck_optimality} for $k=d=4$ we obtain the desired bound.

\paragraph{Case $\ell = 6$.} 

Since any $(4,\, 6)$-cleared graph does not contain $\CC{3}$, all copies of $\CC{4}$ in $G$ must be induced. We shall use the same framework as in the proof of Lemma \ref{lem:ck_optimality}, but with a~different definition of a~good sequence and sets $A_i(D)$. More precisely, let $v_0v_1v_2v_3$ be a copy of $\CC{4}$. We consider good sequences of the form $D = (z_i)_{i=0}^3$, where $z_0 = v_0$, $z_1 = v_2$, $z_2 = v_1$, and $z_3 = v_3$, hence sets $A_i(D)$ are defined as follows:
\begin{align*}
A_0(D) & = V(G),\\
A_1(D) & = N^+_2(z_0) \cap N^-_2(z_0),\\ 
A_2(D) & = N^+(z_0) \cap N^-(z_1),\\
A_3(D) & = N^-(z_0) \cap N^+(z_1).
\end{align*}
Furthermore, we denote $w(D) = \prod_{i=0}^3 |A_i(D)|^{-1}$.

Similarly as in the proof of Lemma \ref{lem:ck_optimality}, for a copy $\mathcal{D}$ of $\CC{4}$ in $G$ on vertices $v_0, v_1, v_2, v_3$, we define $D_j = (v_j, v_{j+2}, v_{j+1}, v_{j+3})$ (with indices taken modulo $4$) as the $j$-th good sequence corresponding to $\mathcal{D}$. Observe that $\sum_\mathcal{D} \sum_{j=0}^3 w(D_j) \leq 1$. Letting $n_{i,j} = \abs{A_i(D_j)}$ and using the same reasoning as in the proof of Lemma \ref{lem:ck_optimality} we obtain
$$
\frac14 \sum_{j=0}^3 w(D_j) \geq \frac1n \left( \frac{1}{12} \sum_{j=0}^3 \sum_{i=1}^3 n_{i,j} \right)^{-3}.
$$

We proceed to the main part of this proof.

\begin{clm}
The following inequality holds
\begin{displaymath}
\sum_{j=0}^3 \sum_{i=1}^3 n_{i,j} \leq 3n,
\end{displaymath}
with equality if and only if each vertex of $G$ is in in-neighborhood and out-neighborhood of two non-adjacent vertices from $\{v_0, v_1, v_2, v_3\}$.
\end{clm}

\begin{proof}
It is enough to show that the contribution of any vertex $w \in V(G)$ to the above sum is at most 3, and that such a~contribution can only occur if $w$ is in in-neighborhood and out-neighborhood of two non-adjacent vertices from $\{v_0, v_1, v_2, v_3\}$.

Assume that $w \in A_1(D_i) \cap A_1(D_j)$ for different $i, j \in \{0,1,2,3\}$. If the difference between $i$ and $j$ is even, then we have a homomorphic image of $\CC{6}$ (Figure \ref{fig:sub1}), which is a~contradiction. Hence, we may assume that $j \equiv i+1 \pmod 4$. If $w \in A_2(D_m)$ or $w \in A_3(D_m)$ for any $m \in \{0,1,2,3\}$, then $w$ is adjacent to $v_i$ or $v_j$, but then we find a~copy of $\CC{3}$ in $G$ (Figure \ref{fig:sub2}). Therefore, the contribution of $w$ in this case is at most $2$.

\begin{figure}[h!]
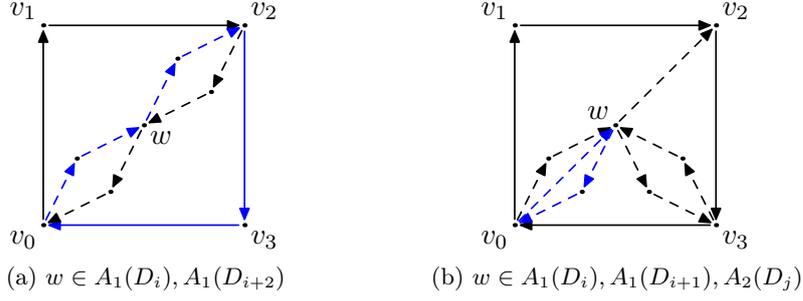

\centering
\begin{subfigure}{.4\textwidth}
  \centering
  \includegraphics[width=.6\linewidth]{drawings/c4_without_c6_1.mps}
  \caption{$w\in A_1(D_i), A_1(D_{i+2})$}
  \label{fig:sub1}
\end{subfigure}
\begin{subfigure}{.2\textwidth}
\end{subfigure}
\begin{subfigure}{.4\textwidth}
  \centering
  \includegraphics[width=.6\linewidth]{drawings/c4_without_c6_2.mps}
  \caption{$w\in A_1(D_i), A_1(D_{i+1}), A_2(D_j)$}
  \label{fig:sub2}
\end{subfigure}
\caption{Forbidden configurations when $w$ belongs to two different sets $A_1$.}
\label{fig:oomx}
\end{figure}

Note also that if $w \in A_2(D_i)$, then $w \notin A_2(D_{i+2})$, as well as $w\notin A_2(D_{i+1}), A_2(D_{i+3})$, because otherwise it creates a directed triangle. Summarizing, $w \in A_1(D_i)$ can happen for at most one $i \in \{0,1,2,3\}$, and the same holds for $A_2$ and, by symmetry, for $A_3$. This implies that the contribution of $w$ can be at most $3$ and the claim follows.
\end{proof}

Combining this result with the previous inequality gives
$$
1 \geq \sum_\mathcal{D} \sum_{j=0}^3 w(D_j) \geq \sum_\mathcal{D} \frac{4}{n} \left( \frac{1}{12} \cdot 3n \right)^{-3} = M \cdot \frac4n \left( \frac n4 \right)^{-3}, 
$$
where $M$ denotes the number of copies of $\CC{4}$ in $G$.
This implies $M \leq \left( \frac{n}{4} \right)^4$.

In conclusion, $\ex(n, \CC{4}, \{\CC{3}, \CC{6}\}) \leq \left( \frac n4 \right)^4$, which implies that $\ex(n, \CC{4}, \CC{6}) = \left(\frac n4 \right)^4 + o(n^4)$.
\end{proof}

\subsubsection{Maximizing $\CC{5}$}

When we maximize the number of copies of $\CC{5}$, more complex extremal constructions can occur than in the previously considered cases. 
This happens when the forbidden cycle is of length $3$, $4$ or $7$. We present solutions with optimal bounds when the forbidden cycle is $\CC{3}$ in Theorem~\ref{thm:ex_c5_c3} and $\CC{7}$ in Theorem~\ref{thm:c5_without_c7}. We also show that for $\ell \neq 3,4,7$ not divisible by $5$, the maximum is asymptotically attained in a balanced blow-up of $\CC{5}$ (Theorem~\ref{thm:ex_c5_cl} and Theorem~\ref{thm:ex_c5_c8}). 

\begin{theorem}\label{thm:ex_c5_cl}
$\ex(n, \CC{5}, \CC{\ell}) = \left (\frac{n}{5} \right)^5 + o(n^5)$ for all $\ell > 5, \; 5 \nmid \ell, \; \ell \neq 7, 8, 13$. 
\end{theorem}

\begin{proof}
    The lower bound is achieved by a balanced blow-up of $\CC{5}$.

    To prove the upper bound, consider a $(5, \ell)$-cleared graph $G$. Using Observation~\ref{obs:cleanup_t3}, existence of a $\TT{3}$ in a $(5, \ell)$-cleared graph implies existence of homomorphic images of $\CC{6}$ and $\CC{9}$, so $\ell$ cannot be of the form $5x + 6y + 9z$ for non-negative integers $x$, $y$, and $z$. It is easy to verify that the only integers larger than $5$ that cannot be represented in this form are $7$, $8$ and $13$. 
    Furthermore, Corollary~\ref{cor:coin_problem} implies that for every $\ell > 5$ other than $7$, any $(5, \ell)$-cleared graph does not contain~$\CC{3}$. 
    
    Therefore, the underlying undirected graph of $G$ does not contain triangles and the number of $\CC{5}$ in $G$ is upper bounded by the number of $C_5$ in a triangle-free undirected graph on $n$ vertices. This maximum was established in \cite{grzesik2012pentagon} and \cite{c5_without_c3}, and matches the desired bound $\left ( \frac{n}{5} \right )^5$.
\end{proof}

\begin{theorem}\label{thm:ex_c5_c3}
$\ex(n,\CC{5},\CC{3}) = \frac{1}{512}n^5 + o(n^5)$
\end{theorem}

\begin{proof}
The lower bound comes from the following construction --- take a balanced blow-up of $\CC{4}$ and add arcs inside each of the blobs so that they become transitive tournaments (see Figure~\ref{fig:constr_c5_c3}). 

    \begin{figure}[ht]
        \centering
        \includegraphics[scale=1.25]{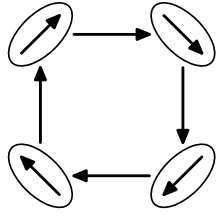}
        \caption{The extremal graph for $\ex(n, \CC{5}, \CC{3})$.}
        \label{fig:constr_c5_c3}
    \end{figure}

The upper bound can be easily obtained using flag algebras. Below, we provide a code for the publicly available software Flagmatic \cite{flagmatic} that can generate and verify the proof.
\end{proof}
\vspace{-2mm}
\begin{lstlisting}[language=Python, numbers=none, breakatwhitespace=true, basicstyle=\small, backgroundcolor=\color{white!98!black}]
from flagmatic.all import *
problem = OrientedGraphProblem(5, forbid="3:122331", density=["5:1223344551","5:122334455113","5:12233445511324"], types=["3:12","3:1223","3:122313"])
problem.set_extremal_construction(field=QQ, target_bound=15/64)
problem.add_sharp_graphs(149,164,168,188,193,265,267,285,312,316)
# type 3:12
problem.add_zero_eigenvectors(0,matrix(QQ, [(0,0,0,0,0,1,0,1,0,0,0,0,0,0,0,0,0,0,0,1,0,0,1,0), (0,0,0,0,0,0,0,-1,0,0,0,1,0,0,0,0,0,0,0,0,0,0,0,0), (0,0,0,0,0,0,0,-1,1,0,0,0,0,0,0,0,0,0,0,0,0,0,0,0), (0,0,0,0,0,-1,1,0,0,0,0,0,0,0,0,0,0,0,0,0,0,0,0,0)]))
# type 3:1223
problem.add_zero_eigenvectors(1,matrix(QQ, [(0,0,0,0,0,0,0,0,1,0,0,0,1,0,1,0,0,0,1,0,0), (0,0,0,0,0,0,0,0,-1,0,0,1,0,0,0,0,0,0,0,0,0), (0,0,0,0,0,0,0,0,0,0,0,0,0,0,0,0,0,0,-1,1,0), (0,0,0,0,0,0,0,0,0,0,0,0,0,0,-1,1,0,0,0,0,0)]))
# type 3:121323
problem.add_zero_eigenvectors(2,matrix(QQ, [(1,0,0,0,0,0,0,0,0,0,0,0,0,0,0,0,2,0,0,1), (0,0,0,0,0,0,0,0,0,0,0,0,0,0,0,0,-1,0,0,1), (0,0,0,0,0,0,0,0,0,0,0,0,0,0,0,0,-1,0,1,0), (0,0,0,0,0,0,0,0,0,0,0,0,0,0,0,0,-1,1,0,0), (0,0,0,0,0,1,0,0,0,0,0,1,0,0,0,0,1,1,0,0), (0,0,0,0,0,0,0,0,0,0,0,0,0,0,0,0,0,-1,0,1), (0,0,0,0,0,0,0,0,0,0,0,0,0,0,0,0,0,-1,1,0), (0,0,0,0,0,0,0,0,0,0,0,0,0,0,0,0,-1,1,0,0), (0,0,1,0,0,0,0,0,0,0,0,0,0,1,0,0,1,0,1,0), (0,0,0,0,0,0,0,0,0,0,0,0,0,0,0,0,0,0,-1,1), (0,0,0,0,0,0,0,0,0,0,0,0,0,0,0,0,-1,0,1,0), (0,0,0,0,0,0,0,0,0,0,0,0,0,0,0,0,-1,1,0,0)]))
problem.solve_sdp()
problem.make_exact()
\end{lstlisting}

\begin{theorem}
\label{thm:c5_without_c7}
$\ex(n, \CC{5}, \CC{7}) = \frac{27}{16} \left (\frac{n}{5} \right)^5 + o(n^5)$. 
\end{theorem}

\begin{proof}
    The lower bound is achieved by the following construction shown in Figure \ref{fig:c5_without_c7}  -- take an \emph{unbalanced} blow-up of $\CC{4}$, where arbitrary two of the blobs have size $\frac{3}{10}n$ each and the two other blobs have size $\frac{2}{10}n$ each. Add arcs inside the two larger blobs to create the balanced complete bipartite graphs with all arcs directed in the same direction. This forms a graph with $6$ blobs distributed in $4$ groups. Each $\CC{5}$ passes through all the groups and $5$ of the $6$ blobs.

    \begin{figure}[h!]
        \centering
        \includegraphics[scale=1.44]{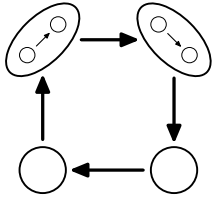}
        \caption{An extremal graph for $\ex(n, \CC{5}, \CC{7})$.}
        \label{fig:c5_without_c7}
    \end{figure}

    To prove the upper bound, consider an extremal $(5,\, 7)$-cleared graph $G$. We say that two vertices $u, v$ of $G$ are \emph{equivalent} if $N^+(u) = N^+(v)$ and $N^-(u) = N^-(v)$. Group vertices of $G$ into equivalence classes (blobs) and let $G'$ be the graph on such blobs (these will correspond to the blobs in our construction) with an arc between the blobs if there is an arc between all vertices in the respective blobs. 

    Clearly, all the vertices in one blob are contained in the same number of $\CC{5}$. Assume that for two different blobs $x$ and $y$ there is no $\CC{5}$ in $G'$ containing both $x$ and $y$ and say that the number of copies of $\CC{5}$ in $G$ containing a vertex in $x$ is at least as large as the number of copies of $\CC{5}$ containing a vertex in $y$. Then by moving all the vertices from $y$ to $x$ we do not decrease the number of copies of $\CC{5}$ in $G$. Therefore, we may assume that for any two blobs, there is a copy of $\CC{5}$ in $G'$ that contains both of them. This implies that any two non-adjacent blobs in $G'$ are connected by a directed path of length $2$ in some direction. 

    Since $G$ is a $(5,\, 7)$-cleared graph, it cannot contain $\Cinv{4}$, as otherwise it contains a homomorphic image of $\CC{7}$. Thus, also $G'$ does not contain $\Cinv{4}$.
    This implies the following claim.
    
    \begin{clm}\label{clm:no_induced_cherries}
    If two blobs have a common out-neighbor or a common in-neighbor, then they are adjacent in $G'$.
    \end{clm}

    \begin{proof}
    Consider a blob $x$ with two out-neighbors $y$ and $z$ in $G'$. If $y$ and $z$ are not adjacent, then there must be a path of length $2$ between $y$ and $z$, without loss of generality, there is a path $y w z$. But then $x,\,y,\,w,\,z$ form a forbidden $\Cinv{4}$ giving a contradiction. Analogous reasoning can be done for two in-neighbors. 
    \end{proof}
 
    This allows us to prove a structural description of $G'$.

    \begin{clm}
    The graph $G'$ forms a $\CC{6}$, possibly with some extra arcs.
    \end{clm}

    \begin{proof}
    Let $C = x_1x_2x_3x_4x_5$ be a~copy of $\CC{5}$ in $G'$. Due to the absence of $\Cinv{4}$, all other arcs among the vertices of $C$ are of form $x_ix_{i+2}$ (indices taken modulo $5$). There can be at most $3$ such arcs, as otherwise a $\Cinv{4}$ is formed. It is straightforward to verify that such a graph contains no $\CC{5}$ other than $C$. Since $G$ contains far more copies of $\CC{5}$ than a blow-up of $\CC{5}$, there is a blob of $G'$ not contained in $C$. Moreover, there is a blob adjacent to a vertex in $C$, because any two blobs appear in a $\CC{5}$.     
    We will show that $G'$ contains a copy of~$\CC{6}$. 
    
    Assume that there is an arc $x_1y$ (the opposite case $yx_1$ is analogous), where $y$ is a blob not contained in $C$. Then $x_2, y \in N^+(x_1)$, so from Claim~\ref{clm:no_induced_cherries}, there is an arc between $x_2$ and $y$. 
    If this arc is $yx_2$, then $x_1yx_2x_3x_4x_5$ forms a~$\CC{6}$. 
    Otherwise, this arc is $x_2y$ and we can use Claim~\ref{clm:no_induced_cherries} once again to deduce that there is an arc between $x_3$ and $y$. Since the arc $x_3y$ forms a forbidden $\Cinv{4}$, this must be the arc $yx_3$. In this case, $x_1x_2yx_3x_4x_5$ forms a $\CC{6}$.
    
    The above reasoning can be also applied to a~copy of~$\CC{6}$ in $G'$ the same way as to a copy of~$\CC{5}$. Therefore, if there is another blob not contained in the proven $\CC{6}$, it extends this $\CC{6}$ to $\CC{7}$, giving a contradiction. 
    Thus, $G'$ consists of $\CC{6}$ possibly with some extra arcs.
    \end{proof}
    
    Let $v_1v_2v_3v_4v_5v_6$ be a $6$-cycle in $G'$. It cannot contain any arc of the form $v_iv_{i+3}$ (indices are taken modulo~6), because then we have the forbidden $\Cinv{4}$. It also cannot contain arcs of the form $v_{i+2}v_i$, because $v_{i+3}$ and $v_i$ would then have a common in-neighbor contradicting Claim~\ref{clm:no_induced_cherries}. Thus, the only possible additional arcs are of the form $v_iv_{i+2}$. Finally, $G'$ cannot contain all three arcs $v_iv_{i+2}$, $v_{i+2}v_{i+4}$, $v_{i+4}v_i$, since then $v_iv_{i+2}v_{i+4}$ is $\CC{3}$, while $v_iv_{i+1}v_{i+2}v_{i+4}$ is $\CC{4}$ and we have a homomorphic image of $\CC{7}$. Thus, there are two arc-maximal constructions, which are shown in Figure~\ref{fig:my_label}.
    Optimizing the sizes of the six blobs of $G$ gives in both cases the same blob sizes as in the conjectured extremal construction and the desired bound for $\ex(n, \CC{5}, \CC{7})$.
\end{proof}

    \begin{figure}[ht]
        \centering
        \includegraphics[width=0.4\textwidth]{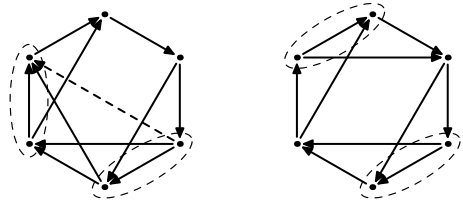}
        \caption{Possible placements of arcs in $G'$. They correspond to adjacent groups (but without the useless dashed arc), or opposite groups in the conjectured extremal construction.}
        \label{fig:my_label}
    \end{figure}

\begin{theorem}\label{thm:ex_c5_c8}
$\ex(n, \CC{5}, \CC{8}) = \left(\frac{n}{5} \right)^5 + o(n^5)$ and $\ex(n, \CC{5}, \CC{13}) = \left(\frac{n}{5} \right)^5 + o(n^5)$. 
\end{theorem}

\begin{proof}
    The lower bound in both cases is achieved by a balanced blow-up of $\CC{5}$. By Observation~\ref{obs:cleanup_small_cycles}, it is enough to prove that \hbox{$\ex(n, \CC{5}, \CC{8}) \leq \left(\frac{n}{5} \right)^5 + o(n^5)$}.
    The proof follows similar lines as the proof of Theorem~\ref{thm:c5_without_c7}. 
    
    Let $G$ be an extremal $(5, 8)$-cleared graph and $G'$ be the graph on equivalence classes (blobs) of $G$. First, recall that $G$ being $(5, 8)$-cleared implies that every arc of $G$ (and consequently $G'$) is contained in some $\CC{5}$.
    Moreover, similarly as in the proof of Theorem~\ref{thm:c5_without_c7}, we may assume that for any two blobs, there is a copy of $\CC{5}$ in $G'$ containing both of them. This implies that any two non-adjacent blobs in $G'$ are connected by a directed path of length $3$ in some direction. 
    Since $G$ is a $(5,\, 8)$-cleared graph, it cannot contain $\Cinv{5}$, as otherwise it contains a homomorphic image of $\CC{8}$. Assume now that two non-adjacent blobs $u$ and $v$ share a common in-neighbor or out-neighbor $w$. Then, a directed path of length $3$ between $u$ and $v$ creates a forbidden homomorphic image of $\Cinv{5}$. Therefore, two non-adjacent blobs cannot share any in-neighbors or out-neighbors.

    Using this fact, we will argue that for any $i\geq 5$, a blob connected to $\CC{i}$ extends it to $\CC{i+1}$. If $x_1x_2\ldots x_i$ is a copy of $\CC{i}$ and there is an arc $x_1y$ in $G'$, then we have an arc $yx_2$ creating $\CC{i+1}$ or an arc $x_2y$. In the latter case, we have an arc $yx_3$ creating $\CC{i+1}$ or an arc $x_3y$. This finally implies that there is an arc $yx_4$ creating $\CC{i+1}$ or an arc $x_4y$ forming a forbidden $\Cinv{5}$ in $G'$.
    Therefore, $G'$ is a $\CC{m}$, possibly with some extra arcs, where $m=5$, $6$ or $7$. We consider those three cases separately. 

    If $m=5$, this $\CC{5}$ must be induced, since $(5,8)$-cleared graphs do not contain $\CC{3}$ and $\CC{4}$. Optimizing the sizes of the blobs shows that a balanced blow-up of $\CC{5}$ contains the largest number of~$\CC{5}$, implying the desired bound. 

    In the case $m=6$, let $v_1v_2v_3v_4v_5v_6$ be the $\CC{6}$ in $G'$. Due to the absence of $\CC{3}$ and $\CC{4}$, the only possible additional arcs are of the type $v_iv_{i+2}$ (indices are taken modulo $6$). Moreover, for any two such arcs $v_iv_{i+2}$ and $v_jv_{j+2}$ we have $j=i\pm1$, as otherwise we have a forbidden $\CC{4}$. In particular, there are at most two such arcs. Without loss of generality, these are $v_1v_3$ and $v_2v_4$. But then there is no $\CC{5}$ containing $v_2v_3$, which contradicts the fact that $G$ is $(5,8)$-cleared. 
    
    For $m=7$, let $v_1v_2v_3v_4v_5v_6v_7$ be the $\CC{7}$ in $G'$. We do not have arcs of the form $v_iv_{i+4}$ (they create $\CC{4}$) or $v_iv_{i+5}$ (they create $\CC{3}$). We will now show that there are also no arcs of the form $v_iv_{i+3}$. Firstly, note that if there are two such arcs sharing a vertex, then we have forbidden~$\CC{3}$. Thus, there are at most $3$ such arcs in total, and so there exists an arc, say $v_1v_{4}$ such that $v_{2}v_{6}$ and $v_{3}v_{7}$ are not in $G'$. The obtained cycle $C = v_{4}v_{5}v_6v_7v_1$ must be induced. 
    The arc $v_1v_{2}$ needs to be contained in some $\CC{5}$, but this $\CC{5}$ must pass either through an arc of the form $v_iv_{i+3}$ (which is impossible since $v_{2}v_{6}$, $v_{3}v_{7}$ are not in $G'$ and other such arcs would be inside $C$) or through two arcs of the form $v_iv_{i+2}$ (also impossible, because one of them would have to be an arc inside $C$).

    We are left with the case where the additional arcs are only of the form $v_iv_{i+2}$. Let $t$ denote the number of these arcs. We consider separate cases depending on the value of $t$, showing in each of them an arc that is not contained in any $\CC{5}$, which is a contradiction.
    \begin{itemize}
        \item Case $t \leq 2$: Let $v_iv_{i+2}$ be one of those arcs. It is easy to verify that the arc $v_iv_{i+1}$ is not contained in any $\CC{5}$.
        \item Case $t = 3$: We say that two arcs $v_iv_{i+2}$, $v_jv_{j+2}$ \emph{intersect} if $i = j \pm 1$. Note that if there is no pair of intersecting arcs, then $G'$ contains forbidden $\CC{4}$. Thus, there is an intersecting pair $v_iv_{i+2}$, $v_{i+1}v_{i+3}$. But then the arc $v_{i+1}v_{i+2}$ is not contained in $\CC{5}$.
        \item Case $t \geq 4$. In this case, there exist two arcs sharing a vertex: $v_{i-2}v_i$ and $v_iv_{i+2}$. This creates a cycle $v_{i-3}v_{i-2}v_iv_{i+2}v_{i+3}$ that must be induced. The only possible remaining arcs are $v_{i-3}v_{i-1}$, $v_{i-1}v_{i+1}$, and $v_{i+1}v_{i+3}$. For the arc $v_{i-2}v_{i-1}$ to be contained in some $\CC{5}$, both $v_{i-1}v_{i+1}$ and $v_{i+1}v_{i+3}$ must be in $G'$. Similarly, for $v_{i+1}v_{i+2}$ to be contained in some $\CC{5}$, both $v_{i-3}v_{i-1}$ and $v_{i-1}v_{i+1}$ must be in $G'$. But then $v_{i-3}v_{i-1}v_{i+1}v_{i+3}$ is a forbidden $\CC{4}$. 
    \end{itemize}
    The contradiction obtained in each case concludes the proof.
\end{proof}

In this subsection, we considered all possible forbidden cycles in the dense case, except when the forbidden cycle is $\CC{4}$. As we argue below, in this case, a different extremal example occurs. We leave finding its exact structure as an open problem.   

    \begin{question}
        What is the value of $\ex(n, \CC{5}, \CC{4})$? 
    \end{question}

    Interestingly enough, this case remains open and without a believable conjecture. An upper bound $\ex(n, \CC{5}, \CC{4}) \leq 0.0567\binom{n}{5}$ may be obtained using Flagmatic. In the following, we provide a (rather complicated) construction that gives a lower bound $\ex(n, \CC{5}, \CC{4}) \geq 0.0517\binom{n}{5}$. Note that among all previously considered constructions that do not contain $\CC{4}$, the best bound is achieved in an iterated blow-up of~$\CC{5}$, but it is only around $0.0385\binom{n}{5}$. 
    
    A far better construction is obtained by splitting vertices equally into $7$ blobs, $V_1, V_2, \ldots, V_7$ and adding all arcs of the form $V_i \rightarrow V_{i+1}$ and $V_i\rightarrow V_{i+3}$ (with indices taken modulo $7$). Note that each of the diagonal arcs $V_i \rightarrow V_{i+3}$ forms a $\CC{5}$ with four $V_i \rightarrow V_{i+1}$ arcs. This construction gives a lower bound of around $0.0499\binom{n}{5}$. Surprisingly, one can improve this even further.

    We tweak the construction described above by flipping direction of some of the $V_i \rightarrow V_{i+1}$ arcs (while leaving orientation of the $V_i \rightarrow V_{i+3}$  arcs the same).
    For arcs between consecutive blobs, we determine orientation as follows. We identify vertices in each blob with points distributed uniformly in the segment $[0, 1]$ and consider a non-decreasing function $f: [0,1] \to [0,1]$ defined as $f(x) = \min\{1,x+c\}$ for some fixed constant $c>\frac{1}{3}$. Now, for two blobs $V_i$ and $V_{i+1}$, we put arc $xy$, for $x \in V_i$, $y \in V_{i+1}$, whenever $f(x) \geq y$ and arc $yx$ if $f(x)<y$. In particular, the out-neighborhood of $x$ in $V_{i+1}$ are the vertices of $V_{i+1}$ in the subset $[0,f(x)]$ and the in-neighborhood of $x$ in $V_{i+1}$ are the vertices in the subset $(f(x), 1]$. 
    
    Observe that the graph induced on $3$ consecutive blobs $V_{i-1} \cup V_i \cup V_{i+1}$ does not contain $\CC{4}$.
    Assuming otherwise, since this graph is bipartite, let $xyx'y'$ be a~copy of $\CC 4$ with $y, y' \in V_i$ and $x,x' \in V_{i-1} \cup V_{i+1}$. Assume $y>y'$. Notice that $N^+(y')\cap V_{i-1} \subset N^+(y)\cap V_{i-1}$. This is because if $f^{-1}$ exists for $y'$, then $f^{-1}(y') \leq f^{-1}(y)$. Additionally $N^+(y')\cap V_{i+1} \subset N^+(y) \cap V_{i+1}$, simply because $f(y') \leq f(y)$. But then $N^+(y') \subset N^+(y)$ and there is no $x$, such that $y' \rightarrow x \rightarrow y$ in $V_{i-1}\cup V_i\cup V_{i+1}$, hence no $\CC{4}$.
    
    Therefore, if there is a $\CC{4}$ in our graph, it has to contain at least one diagonal arc. It is straightforward to verify that such a $\CC{4}$ would have to contain exactly one such arc, denote it by $v_iv_{i+3}$. This leads to a $\CC{4}$ on $4$ consecutive blobs, $v_{i+3}v_{i+2}v_{i+1}v_i$. Recall that $v_{j+1} > f(v_j)$ for $j \in \{i,i+1,i+2\}$. For our choice of $f(x) = \min\{1,x+c\}$ with $c > \frac 13$, it means that $v_{i+1} > \frac 13$, $v_{i+2} > \frac 23$ and $v_{i+3} > 1$, which is impossible. 

    

    Note that the construction has many $\CC{5}$ as there are relatively many paths on $3$ edges in two consecutive blobs. Taking two diagonal arcs merges any such path into $\CC{5}$. The choice of $c$ allows to retain most of $\CC{5}$ in the previous construction and add new ones using diagonal edges. While there is technically a third way to obtain $\CC{5}$ in this construction, using diagonal arcs three times and special reverse arcs in consecutive blobs twice, optimal choice of $c$ does not admit such $\CC{5}$.
    The stated bound $\ex(n, \CC{5}, \CC{4}) \geq 0.0517\binom{n}{5}$ is achieved by considering function $f(x) = \min\{1, x + c\}$ for $c \approx 0.67757$.

\section{The sparse case}\label{sec:sparse_case}

Lastly, we present one example for the sparse case.

\begin{theorem}
$\ex(n, \CC{3}, \CC{6}) = \frac{1}{4}n^2 + o(n^2)$.
\end{theorem}

\begin{proof} 
The lower bound is achieved by the following construction, shown in Figure \ref{fig:constr_c3_c6}. Isolate one vertex $v$ and split remaining vertices into two parts $A$ and $B$ of roughly the same size. Add all arcs from $v$ to $A$, from $A$ to $B$, and from $B$ to $v$. 

\begin{figure}[h!]
    \centering
    \includegraphics{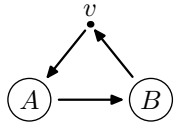}
    \caption{The extremal graph for $\ex(n, \CC{3}, \CC{6})$.}
    \label{fig:constr_c3_c6}
\end{figure}

Let $G$ be an extremal oriented graph on $n$ vertices. We may assume that each arc of $G$ is contained in at least one copy of $\CC{3}$. We say that an arc of $G$ is \emph{thin} if it is contained in exactly one $\CC{3}$, otherwise we say that it is \emph{thick}.

The crucial observation is that each copy of~$\CC{3}$ in $G$ contains at least one thin arc. Indeed, if there is a~copy of~$\CC{3}$ consisting of only thick arcs $vw$, $wu$, and $uv$, then we can find three more vertices $v', w', u'$ in $G$ such that $vw'u$, $uv'w$, and $wu'v$ are directed triangles, which results in a~forbidden $\CC{6}$ $vw'uv'wu'$.

Since each copy of~$\CC{3}$ contains at least one thin arc and each such arc cannot belong to more copies of~$\CC{3}$, we conclude that $\ex(n, \CC{3}, \CC{6})$ is upper bounded by the number of thin arcs in~$G$.

Let $vw$ be a thin arc. Observe that the following configurations are not subgraphs of~$G$:

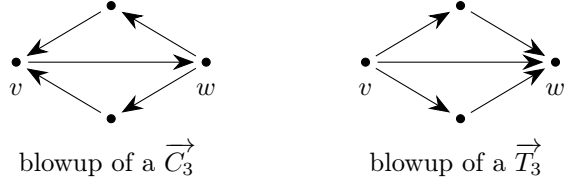
\tikzset{vertex/.style = {shape=circle,draw,minimum size=0.3em, fill=black,inner sep=0pt, outer sep=0.3em}}
\tikzset{edge/.style = {->, >={Stealth[length=0.8em]}}}

\begin{center}
\begin{tikzpicture}[scale=2.5]
\node[vertex] (v) at (0,0) [label=below:$v$] {};
\node[vertex] (w) at (1,0) [label=below:$w$]{};
\node[vertex] (x) at (0.5,0.3) {};
\node[vertex] (y) at (0.5,-0.3) {};
\foreach \from/\to in {v/w, w/x, w/y, x/v, y/v}
    \draw[edge] (\from) to (\to);
\node [below=0.1em of y, align=flush center]
        {
            blowup of a $\CC{3}$
        };
\end{tikzpicture} $\qquad \qquad$
\begin{tikzpicture}[scale=2.5]
\node[vertex] (v) at (0,0) [label=below:$v$] {};
\node[vertex] (w) at (1,0) [label=below:$w$]{};
\node[vertex] (x) at (0.5,0.3) {};
\node[vertex] (y) at (0.5,-0.3) {};
\foreach \from/\to in {v/w, x/w, y/w, v/x, v/y}
    \draw[edge] (\from) to (\to);
\node [below=0.1em of y, align=flush center]
        {
            blowup of a $\TT{3}$
        };
\end{tikzpicture}
\end{center}

First configuration is not allowed just by the definition of a~thin arc. In the second configuration, for each upper arc one can find a copy of $\CC{3}$ containing it, to form a~copy of $\CC{6}$.

Let $G'$ be the subgraph of $G$ consisting of all thin arcs of $G$. 
Applying the regularity lemma to $G'$, we conclude that its reduced graph does not contain any triangles, so by removing $o(n^2)$ arcs we can delete all triangles in $G'$. But then, Mantel's theorem applied to the undirected underlying graph of $G'$ implies that $G'$ have at most $\frac{1}{4}n^2$ arcs.
Therefore, $G$ has $\frac{1}{4}n^2 + o(n^2)$ thin arcs as desired. 
\end{proof}

We believe that if a cycle of length $3t$ is forbidden, then a similar construction with $t-1$ vertices in one blob is extremal up to a lower order error term.

\begin{conj}
For any integer $t \ge 2$, $\ex(n, \CC{3}, \CC{3t}) = \frac{t-1}{4}n^2 + o(n^2)$.
\end{conj}

\section{Directed graphs}\label{sec:directed}

An analogous problem can be also considered in the setting of directed graphs, that is, when we additionally allow vertices to be connected by \emph{digons} -- pairs of vertices $u$ and $v$ connected simultaneously by arcs $uv$ and $vu$. 

In this setting, the maximum number of $\CC{k}$ in an $n$-vertex directed graph not containing $\CC{\ell}$ is also of order $\Theta(n^k)$ if $k \nmid \ell$ and $\Theta(n^{k-1})$ if $k \mid \ell$, because the proof of Theorem~\ref{thm:order_of_magnitude} does not use the assumption that the extremal directed graph does not contain digons. 

Additionally, digons allow to consider a balanced blow-up of $\CC{2}$, that is a complete balanced bipartite digraph, which is an optimal extremal constructions when $k$ is even and $\ell$ is odd. This leads to a more natural definition of $d$ than in Theorem~\ref{thm:general_thm} and the following analog of Theorems~\ref{thm:general_thm} and \ref{thm:rand_thm} for directed graphs.

\begin{theorem}\label{thm:directed}
Let $k$ and $\ell$ be integers satisfying $k \geq 3$, $\ell \geq 2(k-1)^2$, and $k \nmid \ell$. 
Denote by $d$ the smallest positive integer that divides $k$ but does not divide $\ell$. Then the maximum number of $\CC{k}$ in an $n$-vertex directed graph not containing $\CC{\ell}$ is equal to $\frac{n}{k}\left(\frac{n}{d}\right)^{k-1} + o(n^k)$.
\end{theorem}

\begin{proof}
The lower bound is achieved by a balanced blow-up of $\CC{d}$. To prove the upper bound, consider an extremal directed graph $G$ and, as before, apply the directed version of Szemer\'edi's Regularity Lemma to remove all directed walks of length $\ell$ by losing $o(n^k)$ copies of $\CC{k}$. We may also assume that each arc is contained in $\CC{k}$. 

If $\ell$ is even, then a digon forms a directed walk of length $\ell$. Therefore, $G$ is an oriented graph and Theorem~\ref{thm:general_thm} implies the wanted bound.  

Consider now the case that both $\ell$ and $k$ are odd, and assume that $G$ contains a digon. Since every vertex is contained in some $\CC{k}$, we obtain intersecting $\CC{k}$ and a closed directed walk of length $2$, which easily creates a closed directed walk of length $\ell$.
Knowing that $G$ is an oriented graph, we obtain the wanted bound by Theorem~\ref{thm:general_thm} as before.

We are left with the case of odd $\ell$ and even $k$, which means that $d=2$. Observe that $G$ does not contain $\TT{3}$, because otherwise we have intersecting $\CC{k}$ and a closed directed walk of length $k+1$, which leads to a closed directed walk of length $\ell$ by Corollary~\ref{cor:coin_problem} as $\ell \ge (k-1)k$. This implies that $e(G) \leq \frac{n^2}{2}$ by the following inductive reasoning. If $G$ contains a digon on vertices $u$ and $v$, then, to avoid $\TT{3}$, for any vertex $w \in V(G\setminus\{u, v\})$ there are at most $2$ arcs between $w$ and $\{u, v\}$, hence $e(G) \leq 2 + 2(n-2) + e(G\setminus\{u, v\}) \leq \frac{n^2}{2}$. While if $G$ contains no digons, then $e(G) \leq \binom{n}{2} \leq \frac{n^2}{2}$. 
Knowing this, we can count the number of copies of $\CC{k}$ in $G$ as follows. Start with an arbitrary arc $v_1v_2$. Then,  choose an arc $v_3v_4$ such that $v_3 \in N^+(v_2)$. Observe that $v_2v_4$ is not an arc, as otherwise it forms a $\TT{3}$. Therefore, the number of possible candidates for $v_3v_4$ is $|N^+(v_2)|\cdot |V(G) \setminus N^+(v_2)| \leq \frac{n^2}{4}$. This reasoning can be continued for $v_5v_6, \ldots, v_{k-1}v_k$. Since each $\CC{k}$ is counted $k$ times in the process, the number of $\CC{k}$ is at most $\frac 1k e(G) \cdot \left(\frac{n^2}{4}\right)^{(k-2)/2} \leq \frac nk \left(\frac{n}{2}\right)^{k-1}$ as wanted.
\end{proof}

\section{Concluding remarks}\label{sec:conclusion}

\paragraph{Other orientations of cycles.}
In this article, we considered only the case where a directed cycle is forbidden, but one can note that the proven theorems imply similar results when a different orientation of a cycle is forbidden. 
For any cycle $C$ we define its \emph{cycle-type} as an absolute value between the number of arcs directed clockwise and the number of arcs directed counterclockwise.  
If a cycle $C$ has cycle-type $\ell \geq 3$, then $\ell$ is equal to the largest length of a directed cycle, to which $C$ has a homomorphism.
Therefore, Observation~\ref{homomorphic_img_removal} implies that $\ex(n,\CC{k},C) \leq \ex(n,\CC{k},\CC{\ell}) + o(n^k)$, and so the proven results also give upper bounds when $C$ is forbidden. 
In particular, since the extremal constructions in Theorems~\ref{thm:general_thm}, \ref{thm:rand_thm}, \ref{thm:ex_C_3_C_l}, \ref{thm:ex_c4_cl}, \ref{thm:ex_c5_cl} and \ref{thm:ex_c5_c8} do not contain any cycles of cycle-type equal to $\ell$, those theorems give sharp bounds (up to a lower order error term) for forbidding $C$. 

\paragraph{Semi-inducibility of alternating cycles.}
Recently, Basit, Granet, Horsley, K\"undgen and Staden asked (\cite[Problem 9.1]{BGHKS25}) for the maximum possible number of color-alternating cycles of length $k$ in a 2-edge-colored complete graph on $n$ vertices, if $2 \mid k$ and $4 \nmid k$. As noted by Chen and Noel, who solved the case $k=6$ in \cite{CN25}, this problem is asymptotically equivalent to determining the maximum possible number of directed cycles of length $k$ in an oriented bipartite graph with $n$ vertices in each part. 
To see this, given a red/blue edge-coloring of the complete graph on $\{1,\ldots,n\}$, create an oriented bipartite graph with bipartition $\{u_1,\ldots,u_n\}$ and $\{v_1,\ldots, v_n\}$, in which we direct $u_i$ to $v_j$ if the edge $ij$ is red and $v_j$ to $u_i$ if $ij$ is blue. Now, each color-alternating cycle corresponds to two directed cycles in the oriented bipartite graph (exchanging all $v_i$ and $u_i$ in a directed cycle gives the same alternating cycle), and each such pair of directed cycles corresponds to an alternating closed walk. Since the number of closed walks is asymptotically the same as the number of cycles, the two problems are asymptotically equivalent. Therefore, Lemma~\ref{lem:knn_rand_optimal} implies that the answer to Basit et al. problem is equal to $\frac{1}{k}\left(\frac{n}{2}\right)^k + o(n^k)$.

\bibliography{refs}
\bibliographystyle{plain}

\end{document}